\documentclass{amsart}
\usepackage[utf8]{inputenc}
\usepackage[english]{babel}
\usepackage{amsmath,esint}
\usepackage{amsfonts, mathrsfs}
\usepackage{amssymb}
\usepackage{dsfont}
\usepackage{amsthm}
\usepackage{hyperref,polynom}
\usepackage{xfrac}
\usepackage{defs}
\usepackage{fonts}
\usepackage{mathtools}
\usepackage{cite}
\usepackage{wrapfig} 
\usepackage[export]{adjustbox} 
\usepackage{todonotes}
\usepackage{mathtools} 
\usepackage{graphicx} 
\usepackage{array} 
\usepackage{enumerate} 
\usepackage{hyperref} 

\usepackage[a4paper,margin=20ex]{geometry}

\newtheorem{theorem}{Theorem}[section]
\newtheorem{proposition}[theorem]{Proposition}
\theoremstyle{definition}

\theoremstyle{plain}
\newtheorem{lemma}[theorem]{Lemma}
\newtheorem{corollary}{Corollary}[theorem]
\theoremstyle{remark}
\newtheorem{remark}[theorem]{Remark}

\makeatletter
\def\namedlabel#1#2{\begingroup
    #2%
    \def\@currentlabel{#2}%
    \phantomsection\label{#1}\endgroup
}
\makeatother

\newcommand{\e}{\mathrm{e}}
\newcommand{\vphi}{\varphi}

\DeclareMathOperator*{\supp}{supp}
\DeclareMathOperator*{\dist}{dist}

\DeclareMathOperator*{\sgn}{sgn}

\newcommand{\intdm}[3]{\displaystyle \int_{#1} #2 \, \mathrm{d}#3}
\newcommand{\iintdm}[5]{\int_{#1}  \int_{#2} #3 \, \mathrm{d}#4 \, \mathrm{d}#5}
\newcommand{\iiintdm}[7]{\int_{#1} \int_{#2} \int_{#3}  #4 \, \mathrm{d}#5 \, \mathrm{d}#6 \, \mathrm{d}#7}
\newcommand{\intdmt}[4]{\displaystyle \int_{#1}^{#2} #3 \, \mathrm{d}#4}

\newcommand{\diffqbu}{ (\bu(\bx)-\bu(\by)) \cdot \frac{\bx-\by}{|\bx-\by|}}

\usepackage[pagewise]{lineno}\nolinenumbers

\newcommand{\bfs}[1]{\boldsymbol{#1}}
\newcommand{\bdx}{\mathbf{x}}

\newcommand{\bdz}{{\bf z}}
\newcommand{\bdy}{{\bf y}}
\newcommand{\bfxi}{\boldsymbol{\xi}}

\usepackage{eurosym}

\everymath{\displaystyle}

\title[A Fractional Korn-type Inequality with Applications] 
{A fractional Korn-type inequality}
\author{James Scott}
\address{Department of Mathematics,
University of Tennessee Knoxville, TN}
\email{jscott66@vols.utk.edu}
\author{Tadele Mengesha}
\address{
Department of Mathematics,
University of Tennessee Knoxville, TN}
\thanks{
This research is supported by the U.S. National Science Foundation grant DMS-1615726. 
}
\email{mengesha@utk.edu}

\subjclass[2010]{46E35,   46E40,  45G15,  35B65,  74B99}
\keywords{Fractional Korn's inequality, Fractional Sobolev spaces, Poisson-type integral, coupled nonlocal equations, self-improving property}

\begin{document}
\maketitle

\begin{abstract}
We show that a class of spaces of vector fields whose semi-norms involve the magnitude of ``directional" difference quotients is in fact equivalent to the class of fractional Sobolev spaces. 
The equivalence can be considered a Korn-type characterization of fractional Sobolev spaces. We use the result to understand better the energy space associated to a strongly coupled system of nonlocal equations related to a nonlocal continuum model via  peridynamics. Moreover, the equivalence permits us to apply classical space embeddings in proving that weak solutions to the nonlocal system enjoy both improved differentiability and improved integrability.
\end{abstract}

\section{Introduction and statement of main results}

The main focus of this paper is to study the function space of vector fields given by 
\[
 \mathcal{X}_{K, p}(\Omega) =\left\{{\bf v}\in L^{p}(\Omega; \mathbb{R}^{d})\, : \, [{\bf v} ]^{p}_{\mathcal{X}_{K,p}} := \int_{\Omega } \int_{\Omega } K(\bdy - \bdx)\left|\frac{({\bf v}(\bdy) - {\bf v} ({\bdx}))}{|\bdy-\bdx|} \cdot \frac{(\bdy -\bdx)}{|\bdy - \bdx|}\right|^{p} \, \mathrm{d}\bdy \, \mathrm{d}\bdx < \infty\right\}
\]
where $\Omega\subset \mathbb{R}^{d}$ is an open subset and the kernel $K(\bdz)$ is a nonnegative function with appropriate integrability. For a particular class of kernels, our main result states that $ \mathcal{X}_{K, p}(\Omega)$ is equivalent to a Sobolev space. We use this identification and classical embedding estimates to obtain Sobolev regularity for solutions to a  strongly coupled system of nonlocal equations with elliptic measurable coefficients. 

For $p = 2,$ the space $ \mathcal{X}_{K, 2}(\Omega)$ has been used in  nonlocal continuum mechanics \cite{Silling2000, Silling2007, Silling2010} where it appears as the energy space corresponding to the peridynamic  strain energy in small strain linear models.  Some basic structural  properties of $ \mathcal{X}_{K, p}(\Omega)$ have already been investigated in \cite{Mengesha-DuElasticity, Mengesha-Du-non, Du-ZhouM2AN}. There it is shown that for any $1\leq p < \infty$, the  space $ \mathcal{X}_{K, p}(\Omega)$ is a separable Banach space with norm $\left(\|{\bf v}\|_{L^{p}}^{p} +  [{\bf v}]^{p}_{\mathcal{X}_{K,p}}\right)^{1/p}$, reflexive if $1 < p < \infty$,  and is a Hilbert space when $p =2$.   Conditions on the kernel $K$ can be imposed so that a Poincar\'e-Korn type inequality holds over subsets that contain no nontrivial zeros of the semi-norm $[\cdot]_{\mathcal{X}_{\rho,p}}$. It is not difficult to see that $[{\bf v}]_{\mathcal{X}_{K,p}} =0$ if and only if ${\bf v}$ is an affine map with skew-symmetric gradient.  These functional analytic properties of the space were used to demonstrate well posednesss of some nonlocal variational problems using the direct method of the calculus of variations, see \cite{Mengesha-Du-non} for more.

As a difference-based function space, it may seem that $ \mathcal{X}_{K, p}(\Omega)$ contains functions with some ``differentiability." This is not in general true, however. Taking a radial $K$ that is compactly supported and with the property that ${K (\bdx)\over |\bdx|^{p}}$ is integrable,  it is shown in \cite{Mengesha-Du-non} that $ \mathcal{X}_{K, p}(\Omega) = L^{p}(\Omega;\mathbb{R}^{d})$. In the event that the space $\mathcal{X}_{K, p}(\Omega)$ is a proper subset of $L^{p}(\Omega;\mathbb{R}^{d})$, the fact that the semi-norm utilizes the ``directional" or ``projected" difference quotient $
 \frac{{\bf v}(\bdy) - {\bf v} ({\bdx})}{|\bdy-\bdx|} \cdot \frac{(\bdy -\bdx)}{|\bdy - \bdx|}$ 
appears to make the space relatively big compared to those that use the full difference quotient. 
 Nevertheless, by averaging the projected difference quotient over enough directions, it is reasonable to think that the semi-norm generated will be comparable with the one that is associated with  the full difference quotient. {\em However, this remains unclear in general}. Finding general conditions  on $K$ and $\Omega$ so that equivalence holds is an open problem, and here we restrict our discussion on the special class of kernels 
\[
K(|{\bfs \xi}|) = {1\over |{\bfs\xi}|^{d + ps-p}}\,, \quad 0 < s< 1\,, \quad 1<p <\infty\,. 
\] 
We denote the corresponding space by $\mathcal{X}_{p}^{s}(\Omega)$. 
These kernels are associated with the fractional Sobolev spaces  $W^{s, p}(\Omega;\mathbb{R}^{d})$ via the Gagliardo semi-norm, where $W^{s, p}(\Omega;\mathbb{R}^{d})$ is given by 
\[
W^{s, p}(\Omega;\mathbb{R}^{d}) := \left\{{\bf v}\in L^{p}(\Omega;\mathbb{R}^{d}): [{\bf v}]_{W^{s,p}}^{p} :=  \int_{\Omega } \int_{\Omega } \frac{|{\bf v}(\bdy) - {\bf v} ({\bdx})|^{p}}{|\bdy-\bdx|^{d+sp}} \, \mathrm{d}\bdy \, \mathrm{d}\bdx < \infty\right\}\,.
\]
{\em The question is now if $\mathcal{X}_{p}^{s}(\Omega)$ is the same as $W^{s, p}(\Omega;\mathbb{R}^{d})$ for these fractional kernels. }

In a recent work \cite{Mengesha-Hardy}, the second author answers the above question in the affirmative for the special case $p=2$, and $\Omega= \mathbb{R}^{d}$ or $\mathbb{R}^{d}_{+}$.  When $p=2$, and $\Omega=\mathbb{R}^{d}$, the question is tractable because both spaces $\mathcal{X}_{2}^{s}(\mathbb{R}^{d})$, and $W^{s, 2}(\mathbb{R}^{d};\mathbb{R}^{d})$ can be characterized by Fourier symbols which made the camparison of norms more straightforward; see \cite{Du-ZhouM2AN}.  For functions defined over the half-space $\mathbb{R}^{d}_{+}$ and vanishing near the hyperplane $x_d = 0$, one can use an appropriate extension operator to control  the semi-norm $[\cdot]_{W^{s,2}(\mathbb{R}^{d}_{+})} $ by the semi-norm $[\cdot]_{\mathcal{X}_{2}^{s}(\mathbb{R}^{d}_{+})}$ of vector fields in the dense class $C_c^{1}(\mathbb{R}^{d}_{+};\mathbb{R}^{d})$.  
 In this paper we extend these results to any $p \in (1,\infty)$ again providing an answer to the question of equivalence of spaces in the affirmative.  Let us introduce the function space 
\begin{equation}\label{zero-bdry-nonlocal-space}
\mathring{\mathcal{X}}^{s}_{p}(\Omega) = \text{Closure of  ${C_{c}^{\infty}(\Omega; \mathbb{R}^{d}) }$ in ${\mathcal{X}}^{s}_{p}(\Omega)$}.  
\end{equation}

 \begin{theorem}[{\bf Fractional Korn's inequality}]\label{theorem-korn}
 For any $s\in (0, 1)$ and $1 < p < \infty$, 
 \[
 \mathcal{X}^{s}_{p}(\mathbb{R}^{d}) = W^{s, p}(\mathbb{R}^{d};\mathbb{R}^{d})\,, \quad \text{and}\quad \mathring{\mathcal{X}}^{s}_{p}(\mathbb{R}^{d}_{+}) = W^{s, p}_{0}(\bbR^{d}_{+};\mathbb{R}^{d})
  \]
  Moreover, there exists a universal constant $C = C(d, p, s)$ such that for all ${\bf f}\in W^{s, p}(\mathbb{R}^{d};\mathbb{R}^{d})$
 \begin{equation}\label{frackorn}
[{\bf f}]_{\mathcal{X}^{s}_{p}} \leq  [{\bf f}]_{W^{s,p}} \leq C [{\bf f}]_{\mathcal{X}^{s}_{p}}\,.
\end{equation}
 \end{theorem}
While the first inequality in \eqref{frackorn} is trivial, the second inequality is the interesting one, as it gives a control of the integral norm of a pointwise larger function by the integral norm of a pointwise smaller function. We call the second inequality a {\em fractional Korn's inequality} for the following reason.  For a smooth vector field ${\bf f}$, the semi-norm $[{\bf f}]_{W^{s,p}}$ uses the full difference quotient which locally behaves as 
 \[
\left|{{\bf f}(\bdy) - {\bf f}(\bdx)\over |\bdy -\bdx|}\right|^{p}\approx \left|\nabla {\bf f}(\bdx) {\bdy - \bdx\over |\bdy - \bdx|}\right|^{p} + O(|\bdy-\bdx|)
\]
while the semi-norm $[{\bf f}]_{\mathcal{X}^{s}_{p}}$ uses the projected difference quotient and locally behaves as 
\[
\left|\frac{{\bf f}(\bdy) - {\bf f} ({\bdx})}{|\bdy-\bdx|} \cdot \frac{(\bdy -\bdx)}{|\bdy - \bdx|}\right|^{p}\approx  \left|\left\langle(\nabla {\bf f}(\bdx))_{sym}  {\bdy - \bdx\over |\bdy - \bdx|},{\bdy - \bdx\over |\bdy - \bdx|}\right\rangle\right|^{p} + O(|\bdy-\bdx|)
\]
where $(\nabla {\bf f}(\bdx))_{sym} = {1\over 2}(\nabla {\bf f}(\bdx)^{T} + \nabla {\bf f}(\bdx))$ is the symmetric part of the gradient matrix. 
The connection between the projected difference quotient and $\grad_{sym}$ runs deeper; multiplying the semi-norm by the proper correcting constant $(1-s)$ it has been shown in \cite{Mengesha}, following the argument in \cite{BBM}, that the space $\mathcal{X}^{s}_{p}(\Omega)$ ``converges" to 
\[
W^{1,p}_{Sym}(\Omega;\mathbb{R}^{d}) := \{ {\bf u}\in L^{p}(\Omega;\mathbb{R}^{d}) : (\nabla {\bf u})_{sym} \in L^{p}(\Omega;\mathbb{R}^{d\times d}) \}
\] as $s\to 1^{-}$. 
This association suggests that $\mathcal{X}_{p}^{s}(\Omega)$ is the fractional analogue  of  $W^{1,p}_{Sym}(\Omega;\mathbb{R}^{d})$. In turn, $W^{1,p}_{sym}(\Omega;\bbR^d)$ is known to coincide with $W^{1, p}(\Omega;\mathbb{R}^{d})$ via the classical Korn's inequality, a fundamental tool in the theory of linearized elasticity; see \cite{Demengel} for a complete proof.  As such, establishing $\mathcal{X}^{s}_{ p}(\Omega) = \mathcal{W}^{s, p}(\Omega;\mathbb{R}^{d})$ in the affirmative amounts to proving a version of Korn's inequality for fractional Sobolev spaces.  

Our proof of Theorem \ref{theorem-korn} makes use of the classical characterization of functions in the fractional spaces in terms of their  Poisson integrals. Given a vector function ${\bf f}$, its Poisson integral is defined as ${\bf u}(\bdx, t) = p_{t}\ast {\bf f}(\bdx)$, where for each $t > 0, $ the function $p_{t}(\bdx)$ is the standard Poisson kernel. 
For $s\in (0,1)$, $1<p <\infty$ it is well-known \cite[Proposition 7', Chapter V]{Stein} that ${\bf f}\in W^{s, p}(\mathbb{R}^{d};\mathbb{R}^{d})$ if and only if 
\[
\int_{0}^{\infty}\left(t^{1-s} \|\partial_{t} {\bf u}(\cdot, t)\|_{L^{p}} \right)^{p} \, {\mathrm{d}t \over t}  < \infty\,.  
\]
Moreover, the semi-norm $[{\bf f}]_{W^{s,p}}$ is equivalent with $\left(\int_{0}^{\infty}\left(t^{1-s} \|\partial_{t} {\bf u}(\cdot, t)\|_{L^{p}} \right)^{p}\, {\mathrm{d}t \over t}\right)^{1/p}$. To prove Theorem \ref{theorem-korn}, we compare the latter semi-norm with that of $[{\bf f}]_{\mathcal{X}^{s}_{p}}$. The key idea is the introduction of a ``Poisson-type" integral ${\bf U}(\bdx, t)$ of a vector field ${\bf f}$. We construct ${\bf U}$ using a convolution with a modified ``Poisson-type" {\em matrix} kernel 
whose components are some linear combination of convolutions of components of the vector field ${\bf f}$. The structure of the Poisson-type kernel reveals that each component of ${\bf U}$ is related with components of the Poisson integral ${\bf u}$ via Riesz transforms leading to the norm relation  
\[
\|\partial_{t} {\bf u}(\cdot, t)\|_{L^{p}} \leq \|\partial_{t} {\bf U}(\cdot, t)\|_{L^{p}} \text{ for all $t> 0$}\,, \quand \int_{0}^{\infty}\left(t^{1-s} \|\partial_{t} {\bf U}(\cdot, t)\|_{L^{p}} \right)^{p} {dt \over t} \leq C(d, p, s) [{\bf f}]^{p}_{\mathcal{X}^{s}_{p}}\,. 
\]  
Combining these inequalities with the characterization of $W^{s,p}(\bbR^d;\bbR^d)$ in terms of Poisson integrals we obtain the equivalence of spaces.   Interestingly, this approach also leads to a characterization of the whole Besov scale  $\Lambda_{p, q}^{s}$ in terms of the newly defined Poisson-type integrals.  These and other related results will be reported elsewhere.

As an application of Theorem \ref{theorem-korn} we show improved Sobolev regularity of weak solutions to the coupled system of nonlocal equations
formally given as 
\begin{equation}\label{MAINEQN}
\mathbb{L}^{s}_{p, \Omega}({\bf u}) := \intdm{\Omega}{ \frac{A(\bx, \by)}{|\bx-\by|^{d+2s}}  |\cD(\bu)(\bx,\by)|^{p-2} {{(\bdx-\bdy) \over |\bdy-\bdx|}\otimes{(\bdx-\bdy)\over |\bdx-\bdy|}} \left({\bf u}(\bdx) - {\bf u}(\bdy)\right)}{\bdy} = {\bf F}(\bf x)\,. 
\end{equation}
In the above $\Omega$ is a bounded subset of $\mathbb{R}^{d}$ for $ d\geq 2$, the functions $\bF$,  ${\bf u} : \mathbb{R}^{d}\to \bbR^{d}$, and  the quantity $\cD(\bu)$ is given by 
\[\cD(\bu)(\bx,\by) := \diffqbu\,. \]
 We also assume that $s\in (0, 1)$, $p\geq 2$, and that $A(\bdx, \bdy)$ is a measurable function such that $\alpha_1\leq A(\bdx, \bdy)\leq \alpha_2$ and symmetric in the sense that $A(\bdx, \bdy) = A(\bdy, \bdx)$ for any $\bdx$, $\bdy\in \mathbb{R}^{d}$.   
Properly speaking, for a given vector field $\bu\in \mathcal{X}^{s}_{p}(\Omega)$, the operator $\bbL^s_{p,\Omega}(\bu)$ is a vector of distributions acting on test functions $\vphi \in C^{\infty}_0(\bbR^d ; \bbR^d)$  via 
\begin{equation}\label{defn-of-operator}
\langle\bbL^s_{p,\Omega} (\bu), \vphi\rangle = \intdm{\Omega}{\intdm{\Omega}{|\cD(\bu)(\bx,\by)|^{p-2} \, \cD(\bu)(\bx,\by) \cdot \cD(\vphi)(\bx,\by) \frac{A(\bx,\by)}{|\bx-\by|^{d+sp}}}{\by}}{\bx}\,.
\end{equation}
Let $\bF \in [\mathcal{X}^{s}_{p}(\Omega)]^{\ast}$, the dual space of $\mathcal{X}^{s}_{p}(\Omega)$, be given. We say $ \bu\in \mathcal{X}^{s}_{p}(\Omega)$ is a {\em weak solution} to the nonlocal system \eqref{MAINEQN} if for all $\vphi \in C^{\infty}_0(\bbR^d ; \bbR^d)$, 
\[
\langle\bbL^s_{p,\Omega} (\bu), \vphi\rangle =\langle \bF, \vphi\rangle
\]
where $\langle\cdot,\cdot \rangle$ is the duality pairing between $[\mathcal{X}^{s}_{p}(\Omega)]^{\ast}$ and $\mathcal{X}^{s}_{p}(\Omega)$.  

For $p=2$, the system of equations given in \eqref{MAINEQN} is closely related to a nonlocal linearized continuum materials model via peridynamics \cite{Silling2000, Silling2007, Silling2010}.   In this case, the leading operator in the system is made up of weighted averages of some linear combinations of vectors of difference quotients. See \cite{Du-ZhouM2AN,Mengesha-Du-non, Mengesha-DuElasticity}  for proper mathematical analysis for the linear case.  The quantity $\frac{\cD(\bu)(\bx,\by)}{|\by - \bx|}$ is what is known as the ``linearized nonlocal strain" and has been used in nonlinear models of damage and fracture \cite{Lipton14,Lipton15} as well. 
For any $1 < p <\infty$ and $A(\by, \bx) = a(|\by-\bx|)$,  by using variational methods well posedness of the coupled system \eqref{MAINEQN} has been established in \cite{Mengesha-Du-non} with appropriate volumetric conditions. Moreover, by exploiting the connection between the spaces $\mathcal{X}^{s}_{p}(\Omega)$ and $W^{1,p}_{Sym}(\Omega;\mathbb{R}^{d})$ it has been shown that \eqref{MAINEQN} is 
 a fractional analogue 
of a strongly coupled nonlinear system of partial differential equations of the type 
\begin{equation}\label{local-eqn}
\text{{\bf div}} \left(\left| (\nabla {\bf u})_{sym}\right|^{p-2}(\nabla {\bf u})_{sym}(\bdx)\right) = {\bf F}(\bdx)\,.
\end{equation}
In fact, for specific variational problems, this relationship has been established via $\Gamma$-convergence in \cite{Mengesha-Du-non}  in the event of vanishing nonlocality (that is, $s\to 1^{-}$).   Regularity of solutions of the nonlinear system has been the subject of recent works, see \cite{Veiga-Crispo}. 

The second main result of the paper is on the self-improving properties of solutions to the nonlocal nonlinear coupled system \eqref{MAINEQN}.  The following is the precise statement we will prove. 
\begin{theorem}\label{Schikorra-theorem}
Let $s\in(0,1)$, $p \geq 2$, and $\Omega\subset \bbR^{d}$ be a bounded domain. Let $\bF \in [\mathcal{X}^{s - \veps(p-1)}_{p}(\Omega)]^{\ast}$, and let $\bu$ be a weak solution to the coupled system of nonlocal equations \eqref{MAINEQN} corresponding to $\bF$. Then there exists a positive constant $\veps_{0}$ such that for all $\veps \in (0, \veps_0)$ the weak solution $\bu$ belongs to  $W^{s+\veps, p}_{ loc} (\Omega)$. Moreover, for any $\eta \in C^{\infty}_{c}(\Omega)$, there exists a positive constant $C$ such that 
\[
[\eta\bu]_{W^{s + \veps, p}(\mathbb{R}^{d})} \leq C \|{\bF }\|^{1\over p-1}_{\left[ \mathcal{X}^{s - \veps(p-1)}_{p}(\Omega) \right]^{\ast}} + C \|\bu\|_{\mathcal{X}^{s}_{p}(\Omega)}\,.
\]
\end{theorem}

The implication  of the regularity result in the theorem is that, with no additional smoothness conditions on the coefficient $A(\bdx, \bdy)$, a  weak solution to the coupled system \eqref{MAINEQN} has improved fractional differentiability in response to improved regularity in the data.   
For scalar equations, this type of self-improving property of solutions is obtained in \cite{Kuusi-M-S} using reverse H\"older inequalities and nonlocal Gehring-type lemmas, obtained in \cite{Schikorra} via a commutator estimate and later obtained in \cite{Auscher} via a functional analytic approach.  The main contribution of this  paper is the extension of the self-improving properties of solutions obtained in the above cited works to the nonlocal nonlinear system \eqref{MAINEQN}.  
We should note that the application of appropriate embedding estimates imply both improved differentiability and higher integrability. For scalar nonlocal equations higher integrability (without improved differentiability) of weak solutions was established in  \cite{Bass-Ren} following classical techniques. The result in  \cite{Bass-Ren} is extended to hold for solutions to the nonlocal system \eqref{MAINEQN} in the recent work \cite{Scott-Mengesha}. 

To prove Theorem \ref{Schikorra-theorem}, we follow the approach in \cite{Schikorra} and is close in spirit with the technique of ``differentiating the equation," and finding relations between higher derivatives of solutions and test functions in order to estimate derivatives of the solution.  This is possible for classical linear equations via  integration by parts and transferring derivatives to test functions. For a special case of the nonlocal system at hand, for $p=2$, $K=1$, and $\Omega = \bbR^d$ we can demonstrate this easily. First notice that we can write  the operator $\bbL^s_{2, \bbR^d}$ in Fourier symbols as 
\[
\cF({\bbL^s_{2, \bbR^d}{\bf u}})(\bfxi) = \big( 2 \pi |\bfxi| \big)^{2s}\left( l_{1}\mathbb{I} + l_2 {\bfxi\otimes \bfxi\over |\bfxi|^{2}}\right) 
\]
where $l_1$ and $l_2$ are positive constants, see \cite{Mengesha-Hardy, Du-ZhouM2AN}.  Then for $\veps>0$ small, via Plancherel's theorem $\langle \bbL^{s+\veps}_{2,\bbR^d} \bu,\vphi\rangle = \langle\bbL^s_{2,\bbR^d} \bu, (-\Delta)^{\veps} \vphi\rangle$, where for any $\alpha$ the operator  $(-\Delta)^{\alpha} \vphi$ is the $\alpha$-fractional Laplacian. When working with the nonlinear ``regional" operator $\bbL^s_{p, \Omega}({\bf u})$, such a clean transfer of derivatives to the test function is not possible. However, as has been done in \cite{Schikorra} one can measure the price of transferring the derivatives by estimating the resulting commutator.  Unlike \cite{Schikorra}, the estimates we establish are based on the smaller $[\cdot]_{\mathcal{X}^{s}_{p}}$ norm, leading us to write some arguments closely resembling those in \cite{Schikorra}. Afterward, we use our first result (Theorem \ref{theorem-korn}) to conclude that the estimates are also valid using the larger semi-norm $[\cdot]_{W^{s,p}}$.

The rest of the paper is organized as follows. In the first part, we focus on proving Theorem \ref{theorem-korn}. To that end, in the next section we recall the classical Poisson kernel and  will present some preliminaries.  We will also review how it is used in the characterization of functions belonging to the fractional Sobolev spaces. In Section \ref{Poisson-type} we introduce a Poisson-type kernel that is central to our result. Its properties as well the relationship between associated Poisson-type integrals and classical Poisson integral will be established. This relationship will be used to in Section \ref{characterization} to prove the main result of the paper. In the second part of the paper we will prove Theorem \ref{Schikorra-theorem}. 
\section{Preliminaries: Poisson integrals and The Riesz transforms}
We recall the classical Poisson kernel and some of its properties that we will use in this paper. We begin with the formula
\begin{equation}\label{Poisson}
p_t(\bx) := \frac{2}{\omega_d} \frac{t}{(|\bx|^2 + t^2)^{\frac{d+1}{2}}}\,,  \quad t > 0\,,  \quad\bdx\in \mathbb{R}^{d}\,, 
\end{equation}
where $\omega_d$ is the volume of the unit sphere in $\mathbb{R}^{d+1}$. 
It is easy to check that $p_t$ is an approximation to the identity.
Its Fourier transform is given by  $\cF({p_t})(\bfxi) = \e^{-2 \pi |\bfxi| t}$ for every $t > 0$, where the Fourier transform operator $\cF$ is given by the formula 
\[
\cF(g)(\bfxi) = \int_{\mathbb{R}^{d}} e^{-\imath2\pi\bfxi\cdot\bdx} f(\bdx)  \, \mathrm{d}\bdx\,. 
\]
It then follows from the Fourier transform expression that the Poisson kernel has the semigroup property $p_{t_1} \, \ast p_{t_2} = p_{t_1+t_2}$ for every $t_1$, $t_2 > 0$.  Using the notation $\nabla$ for the vector of partial derivative operators  $(\nabla_{\bdx}, \partial_{t}) = (\p_1, \p_2, \ldots , \p_{d},\p_t)$, we have that  $\nabla p_{t} \in L^{1}(\mathbb{R}^{d})$ with the estimates 
\[
\int_{\mathbb{R}^{d}} |\partial_{t} p_{t}(\bdx)|\, \mathrm{d}\bdx \leq  {c\over t}\,, \qquad \int_{\mathbb{R}^{d}} |\partial_{x_{j}} p_{t}(\bdx)| \, \mathrm{d}\bdx \leq {c\over t}\,, \qquad  t>0\,, \quad  j=1,\cdots, d\,,
\]  
for some constant $c > 0$. 
For any $f\in L^{p}$, $1\leq p \leq \infty$, its \textit{Poisson integral} is given by 
\[u(\bx,t) := p_t \, \ast f (\bx) = \int_{\mathbb{R}^{d}} p_t(\bdy) f(\bdx-\bdy)\, \mathrm{d}\bdy\,. 
\] 
The Poisson integral is a $C^{\infty}$ harmonic function in $\mathbb{R}^{d+1}_{+} := \bbR^d \times (0,\infty)$, 
with the property that  ${ u} (\cdot, t) \to f$ in $L^{p}(\mathbb{R}^{d})$ as $t\to 0$.  For a vector field ${\bf f}$ its vector-valued Poisson integral ${\bf u} (\bdx, t) = p_t * {\bf f}(\bdx)$ will be defined where the convolution is taken component wise.

The Riesz transforms will be used frequently throughout this work.  We recall that for $1 \leq j \leq d$ and $f \in \cS(\bbR^d)$ the class of Schwartz functions, the $j^{th}$ \textit{Riesz transform} is an operator defined as
\[
R_j(f)(\bx) := \frac{2}{\omega_d} \pv \int_{\bbR^d}{\frac{y_j}{|\by|^{d+1}} f(\bx-\by) }{\, \mathrm{d}\by}\,.
\]
For any $f \in \cS(\bbR^d)$ we have
$
\cF(R_j(f))(\bfxi) = -\imath \frac{\xi_j}{|\bfxi|} \cF(f)(\bfxi)\,.
$
From this formula it is immediately clear that the Riesz transforms commute with partial differential operators $\p_{x_j}$. We recall also the celebrated result of $L^p$ boundedness (c.f. \cite[Chapter III]{Stein}), namely
$$
\Vnorm{R_j f}_{L^p(\bbR^d)} \leq C(p) \Vnorm{f}_{L^p(\bbR^d)}\,, \qquad j = 1, \ldots, d\,, \quad 1 < p < \infty\,.
$$
The Riesz transforms can be used to establish relations between the partial derivatives of functions. Let us display such relations for the Poisson integral of a function now. First note that for $\bff \in \cS(\bbR^d;\bbR^d)$, its Poisson integral $u = p_{t}\ast {\bf f}$ belongs to $\cS(\bbR^d; \bbR^d)$. Further, for any $\bdx\in \mathbb{R}^{d}$ and $t > 0$ we have 
\begin{equation}\label{normal-tangential-PI}
 \p_t \bu(\bx,t) = - \sum_{j=1}^d R_j (\p_{x_j} \bu)(\bdx,t)\, ,\qquad \p_{x_j} \bu(\bdx,t) = R_j ( \p_t \bu)(\bdx,t) \quad \text{ for } j=1, \cdots, d\,.
\end{equation}
We can verify the above identities by taking the Fourier transform in the $\bdx$ variable as follows.  
\begin{align*}
\cF_{\bdx} \big( \p_t \bu(\cdot,t) \big)(\bfxi) = \p_t \, \e^{- 2 \pi |\bfxi| t}\cF({\bff})(\bfxi) &= - 2 \pi |\bfxi| \e^{-2  \pi |\bfxi| t} \cF({\bff})(\bfxi) \\
& = - \sum_{j=1}^d \left( - \frac{\imath \xi_j}{|\bfxi|} \right)  (2 \pi \imath \xi_j) \e^{-2 \pi |\bfxi| t}\cF({\bff})(\bfxi) \\
&= \cF \left( - \sum_{j=1}^d R_j \big[ \p_{x_j} \bu(\cdot,t) \big] \right)(\bfxi)\,, 
\end{align*}
demonstrating the first relation in \eqref{normal-tangential-PI}.  Conversely, 
\begin{align*}
\cF_{\bdx} \big( \p_{x_j} \bu(\cdot,t) \big)(\bfxi) = (2 \pi \imath \xi_j) \e^{- 2 \pi |\bfxi| t} \cF({\bff})(\bfxi) &= \left(-\imath \frac{\xi_j}{|\bfxi|} \right)  (- 2 \pi |\bfxi|) \e^{-2  \pi |\bfxi| t} \cF({\bff})(\bfxi) \\
&= \left(-\imath \frac{\xi_j}{|\bfxi|} \right) \p_t \Big( \e^{-2 \pi |\bfxi| t} \cF({\bff})(\bfxi)  \Big)= \cF_{\bdx} \Big( R_j \big[ \p_t \bu(\cdot,t) \big] \Big)(\bfxi)\,, 
\end{align*}
establishing the second identity in \eqref{normal-tangential-PI}. The pointwise relation in \eqref{normal-tangential-PI} and the $L^p$ boundedness of the Riesz transforms implies that for every Schwartz vector field $\bff \in \cS(\bbR^d;\mathbb{R}^{d})$ and $1 < p < \infty$,
\begin{equation}\label{eq-ComparisonOfDerivativesOfPoissonIntegral-FullBesovSpaces}
\Vnorm{\p_t \bu(\cdot,t)}_{L^p(\bbR^d)} \approx \sum_{j=1}^d \Vnorm{\p_{x_j} \bu(\cdot,t)}_{L^p(\bbR^d)}
\end{equation}
where $ \approx $ represents equivalence of norms up to a constant independent of ${\bf f}$.  Using density of $\cS(\bbR^d;\mathbb{R}^{d})$ in $L^{p}(\bbR^d;\mathbb{R}^{d})$ and the fact that $|\nabla p_{t}| \in L^{1}(\mathbb{R}^{d})$, \eqref{eq-ComparisonOfDerivativesOfPoissonIntegral-FullBesovSpaces} remains true for all ${\bf f} \in L^{p}(\bbR^d;\mathbb{R}^{d})$.

Poisson integrals  can be used to  give a characterization of the $L^{p}$ norm of a function;  see \cite[Chapter IV]{Stein} for details.  
Given a function $f$ we introduce the Littlewood-Paley $g$-function of  $f$ in terms of its Poisson integral $u$ as 
\[
g(f)(\bdx) = \left(\int_{0}^{\infty}t\, |\nabla u(\bdx, t)|^{2} \, \mathrm{d}t \right)^{1/2}\,, \qquad \nabla u(\bdx, t) = (\nabla_{\bdx} u, \partial_{t} u)\,. 
\]
Theorem 1 of \cite[Chapter IV]{Stein} states that for $1 < p < \infty$, if $f\in L^{p}(\mathbb{R}^{d})$ so is $g(f)$, and its $L^{p}$ norm is comparable with that of $f$.
Most important to our work is the usefulness of Poisson integrals in characterizing fractional Sobolev spaces $W^{s,p}(\mathbb{R}^{d};\mathbb{R}^{d})$. 
\begin{proposition}[Proposition $7'$ in\cite{Stein}]\label{thm-Prop7InStein}
Let $s\in (0, 1)$ and $1 <p < \infty$. Let $\bff = (f_1,f_2, \ldots, f_d) \in L^p(\bbR^d)$. Then $\bff \in W^{s, p}(\bbR^d;\mathbb{R}^{d})$ if and only if
\begin{equation}\label{eq-DyConditionOnPoissonIntegral-FullBesovSpace}
\int_{0}^{\infty}{ t^{p(1-s)} \Vnorm{\p_t \bu(\cdot,t)}^{p}_{L^p(\bbR^d)} \, \frac{\mathrm{d}t}{t}} < \infty\,. 
\end{equation}
Moreover, there exists constants $C_{1}$ and $C_{2}$ depending only on $s$, $p$, and $d$ such that 
\begin{equation} \label{semi-norm-EQUIV}
C_{1}\int_{0}^{\infty}{ t^{p(1-s)} \Vnorm{\p_t \bu(\cdot,t)}^{p}_{L^p(\bbR^d)} \, \frac{\mathrm{d}t}{t}} \leq [{\bf f}]_{W^{s, p}}^{p} \leq C_2 \int_{0}^{\infty}{ t^{p(1-s)} \Vnorm{\p_t \bu(\cdot,t)}^{p}_{L^p(\bbR^d)} \, \frac{\mathrm{d}t}{t}}\,.
\end{equation}

\end{proposition}

\begin{proof}
The inequality on the left-hand side  in \eqref{semi-norm-EQUIV} is proved in \cite{Stein, Taibleson}, and the right-hand proved in \cite{Taibleson}. However the inequality on the right-hand side is the one that we need later and so for completeness 
we present its proof here. We will prove it for scalar functions, and for the vector case it follows easily by making the comparison component wise. We let $u(\bdx, t)= p_{t}\ast {f} (\bdx)$.  
Let $\bx$, $\by$ be such that $\bx$ and $\bx + \by$ are Lebesgue points of $f$. We choose $t = |\by|$ and write
\begin{equation}
\begin{split}
	f(\bx+\by) - f(\bx) &= \Big( u(\bx+\by,|\by|) - u(\bx,|\by|) \Big) \\
	&+ \Big( f(\bx+\by)- u(\bx+\by,|\by|) \Big) - \Big( f(\bx)-u(\bx,|\by|) \Big)\,.
\end{split}
\end{equation}
We estimate each of the integrals associated with the three differences separately. We denote these integrals by $\mathrm{I}, \mathrm{II}$ and $\mathrm{III}$.  Using the mean value theorem, 
\begin{equation}\label{eq-FundamentalTheoremOfCalculusIdentity}
	u(\bx+\by,|\by|) - u(\bx,|\by|) = \int_{0}^{1}{\grad_{\bx} u(\bx+\tau \by,|\by|) \cdot \by}{\, \mathrm{d}\tau}\,.
\end{equation}
It then follows from Minkowski's inequality that 
\begin{equation}
\begin{split}
\left( \int_{\bbR^d}{\Big| u(\bx+\by,|\by|) - u(\bx,|\by|) \Big|^p}{\, \mathrm{d}\bx} \right)^{1/p} &=  \left( \int_{\bbR^d}{\left| \int_{0}^{1}{\grad_{\bx} u(\bx+\tau \by,|\by|) \cdot \by}{\, \mathrm{d}\tau} \right|^p}{\, \mathrm{d}\bx} \right)^{1/p} \\
&\leq |\by| \, \Vnorm{\grad_{\bx} u(\cdot,|\by|)}_{L^p(\bbR^d)}\,.
\end{split}
\end{equation}
Then using polar coordinates ($t = |\by'|$) we get that
\begin{equation}\label{eq-EstimateForPartII-I}
\begin{split}
	\left(\mathrm{I}\right)^{p} :=  \int_{\bbR^d}\int_{\bbR^d} \frac{\left| u(\bx+\by,|\by|) - u(\bx,|\by|) \right|^p}  {|\by|^{d+sp}}{\, \mathrm{d}\bx}  \, \mathrm{d}\bdy
	&\leq  \int_{\bbR^d}{\frac{  \Vnorm{\grad_{\bx} u(\cdot,|\by|)}^{p}_{L^p(\bbR^d)} }{|\by|^{d+sp-p}}}{\, \mathrm{d}\by}  \\
	&\leq C \int_{0}^{\infty}{\frac{\Vnorm{\grad_{\bx} u(\cdot,t)}^{p}_{L^p(\bbR^d)} }{t^{sp-p+1}}}{\mathrm{d}t}\,.
\end{split}
\end{equation}
Now we repeat the same argument for the second difference; using \eqref{eq-FundamentalTheoremOfCalculusIdentity}, 
\[
	f(\bx+\by) - u(\bx+\by,|\by|) = - \int_{0}^{1}{\p_t u(\bx+ \by, \tau |\by|) \, |\by|}{\, \mathrm{d}\tau}\,.
\]
Then Minkowski's inequality gives us 
\begin{equation}
\begin{split}
\left( \int_{\bbR^d}{\Big| f(\bx+\by) - u(\bx+\by,|\by|) \Big|^p}{\mathrm{d}\bx} \right)^{1/p} 
&\leq \int_{0}^{1}{|\by| \, \Vnorm{\p_t u(\cdot,\tau |\by|)}_{L^p(\bbR^d)}}{\mathrm{d} \tau}\,.
\end{split}
\end{equation} Calculations similar to the one above along with a second application of Minkowski's inequality show that 
\[
	\mathrm{II} := \left( \intdm{\bbR^d}{\frac{ \Vnorm{f(\cdot +  \by) - u(\cdot + \by,|\by|) }^{p}_{L^p(\bbR^d)} }{|\by|^{d+sp}}}{\by} \right)^{1/p}
	 \leq C \intdmt{0}{1}{\left( \intdmt{0}{\infty}{\frac{ \Vnorm{\p_t u(\cdot,\tau r)}^{p}_{L^p(\bbR^d)}}{r^{sp-p+1}}}{r} \right)^{1/p}}{\tau}\,.
\]
Changing variables $t = \tau r$ in the inner integral we obtain that 
\begin{equation}\label{eq-EstimateForPartII-II}
\begin{split}
\mathrm{II} \leq C \int_{0}^{1} \left( \int_{0}^{\infty}{\frac{ \Vnorm{\p_t u(\cdot,t)}^p_{L^p(\bbR^d)} }{t^{sp-p+1}} \tau^{sp-p}}{\, \mathrm{d}t} \right)^{1/p} {\, \mathrm{d}\tau}  &\leq C \int_{0}^{1}{\tau^{s-1}}{\, \mathrm{d}\tau} \left( \int_{0}^{\infty}{\frac{ \Vnorm{\p_t u(\cdot,t)}^p_{L^p(\bbR^d)} }{t^{sp-p+1}} }{\, \mathrm{d}t} \right)^{1/p}  \\
	&= \frac{C}{s}  \left( \int_{0}^{\infty}{\frac{ \Vnorm{\p_t u(\cdot,t)}^p_{L^p(\bbR^d)}}{t^{sp-p+1}}}{\, \mathrm{d}t} \right)^{1/p}\,.
\end{split}
\end{equation}
The quantity $f_j(\bx) - u(\bx,|\by|)$ can be estimated exactly the same way, and so we obtain
\begin{equation}\label{eq-EstimateForPartII-III}
\begin{split}
	\left(\mathrm{III}\right)^{p} &:=  \int_{\bbR^d}{\frac{ \Vnorm{f - u(\cdot,|\by|) }^p_{L^p(\bbR^d)}}{|\by|^{d+sp}}}{\mathrm{d}\by}  \leq  {C \over s^p} \int_{0}^{\infty}{\frac{\Vnorm{\p_t u(\cdot,t)}^p_{L^p(\bbR^d)}}{t^{sp-p+1}}}{\, \mathrm{d}t} \,. 
\end{split}
\end{equation}
We now invoke the comparison estimate \eqref{eq-ComparisonOfDerivativesOfPoissonIntegral-FullBesovSpaces} to conclude the proof. 
\end{proof}

\section{Poisson-type integrals}\label{Poisson-type}
In this section we introduce a Poisson-type matrix kernel $\mathbb{P}_{t}(\bdx)$   that we convolve with vector fields so that the resulting Poisson-type integral can be used to characterize $\mathcal{X}^{s}_{p}(\mathbb{R}^{d})$ 
in the same spirit as Proposition \ref{thm-Prop7InStein}.. 
\subsection{Definition of Poisson-type kernel and integral}
\subsubsection*{Poisson-type kernel} Denoting $ \bbM_{J}(\bbR)$ to be the space of $J\times J$ matrices with real entries, we introduce the  $\bbM_{d+1}(\bbR)$-valued function $\mathbb{P}(\bdx) = \left(\mathfrak{p}^{jk}(\bdx)\right)_{j,k=1}^{d+1}$ as 
\begin{equation}\label{sym-oned-poisson}
\mathfrak{p}^{jk}(\bx)  := \frac{2(d+1)}{\omega_d} \frac{(\bx,1)_j (\bx,1)_k}{(|\bx|^2 + 1)^{\frac{d+3}{2}}}\,, \quad \bdx \in \mathbb{R}^{d}, \quad j, k=1,\cdots, d+1 
\end{equation}
where $(\bdx, 1)$ is a $d+1$ vector whose $d+1$ component is 1. 
For $t > 0$, we denote the the Poisson-type kernel $\mathbb{P}_t : \bbR^d  \to \bbM_{d+1}(\bbR)$ by
\begin{equation}
\mathbb{P}_t(\bx) := {1\over t^{d}} \mathbb{P}\bigg({\bx\over t}\bigg)\,, \,\, \bdx\in \mathbb{R}^{d}.
\end{equation}
We notice that  the $(d+1)\times (d+1)$ matrix $\mathbb{P}_t(\bdx)$ has the form
\begin{equation} \label{sym-oned-poisson-parts}
\mathbb{P}_t(\bx) = 
	\frac{2(d+1)}{\omega_d} \frac{t}{(|\bx|^2 + t^2)^{\frac{d+3}{2}}}  \begin{bmatrix}
	 \bx \otimes \bx& t\bdx \\
	t\bdx & t^{2} \\
	\end{bmatrix}\,,
\end{equation}
where $\bdx$ is considered both a row and column vector in $\bbR^{d}$. 

\subsubsection*{Poisson-type integrals} 
Given $\bF : \bbR^{d} \to \bbR^{d+1}$ with $\bF = (F_1, \dots, F_{d+1}) \in L^p(\bbR^d)$, we define $\bU = (U_1, U_2, \ldots, U_{d+1}) : \bbR^{d+1}_+ \to \bbR^{d+1}$ as the convolution
\begin{equation}\label{Sym-Poisson}
\bU(\bx,t) := \mathbb{P}_t\, \ast \, \bF (\bx)\,.
\end{equation}The convolution in the above equation is taken in the sense of matrix multiplication. That is the $i^{th}$ entry component of $\bU$ is given by $U_i = \sum_{j=1}^{d + 1} \mathfrak{p}_t^{ij}\ast F_{j}$.

\subsection{Properties of the Poisson-type Kernel} 
We next establish basic but fundamental properties of $\mathbb{P}_t$  which are analogues of the properties of the classical Poisson kernel.  We begin by noting that  $\mathbb{P} \in C^{\infty}(\bbR^d;\bbM_{d+1}(\bbR))$, and
\[
(\bx,t) \mapsto \mathbb{P}_t(\bx) \in C^{\infty}\left( \overline{\bbR^{d+1}_+} \setminus B(0,\veps)\, ;\, \bbM_{d+1}(\bbR) \right)\,, \quad  \text{for every }\veps > 0\,.
\]
Moreover, it is immediate from the definition to see that for every $1 \leq p \leq \infty$, $\mathbb{P} \in L^p(\bbR^d ;\bbM_{d+1}(\bbR))$ with the pointwise estimate \[
|\mathbb{P}(\bx)| \leq \frac{C}{(1+|\bx|^2)^{\frac{d+1}{2}}}\,, \quad \bx \in \bbR^d\,.\]
By the norm of $|\mathbb{A}|$ for a matrix $\mathbb{A} = (a^{ij})$ we mean the Frobenius norm $|\mathbb{A}|  = \sqrt{\sum_{i,j} |a^{ij}|^{2}}$. 
In the following lemma we prove that the matrix kernel {$\mathbb{P}_{t}$} is in fact  an approximation to the identity. 
\begin{lemma}\label{thm-PropertiesOfPoissonTypeKernelP}
For any $t>0$, if $\mathbb{I}_{d+1}$ denotes the $(d+1)\times (d+1)$  identity matrix, then  
\[\int_{\bbR^d}{\mathbb{P}_{t}(\bx)}{\, \mathrm{d}\bx} =\int_{\bbR^d}{\mathbb{P}(\bx)}{\, \mathrm{d}\bx}  = 
  \bbI_{d+1}\,.\]   For any $\bF  \in L^1_{loc}(\bbR^d;\mathbb{R}^{d+1})$, the function $(\bx,t) \mapsto (\mathbb{P}_t \, \ast \bF)(\bx) \in C^{\infty}(\bbR^{d+1}_+)$ with
\begin{equation}
\lim\limits_{t \to 0^+} (\mathbb{P}_t \, \ast \, \bF)(\bx) = \bF(\bx) \quad \text{ for all Lebesgue points } \bx \in \bbR^d \text{ of } \bF\,.
\end{equation}
Moreover,  for $1 \leq p < \infty$ if $\bF \in L^p(\bbR^d;\mathbb{R}^{d+1})$, then 
\begin{equation}
\lim\limits_{t \to 0^+} \Vnorm{(\mathbb{P}_t \, \ast \, \bF) - \bF}_{L^p(\bbR^d)} \to 0\,. 
\end{equation}
\end{lemma}

\begin{proof} The conclusion of the lemma can be deduced from \cite[Lemma 3.3]{Martell-M32016} where it is shown that similar properties are enjoyed by the Poisson kernel for the Lam\'e system. To be specific, given constants $\mu$, $\lambda$ satisfying $3 \mu + \lambda > 0$, $\mu + \lambda > 0$, the matrix kernel  $\bbK : \bbR^d \to \bbM_{d+1}(\bbR)$ defined by
$$\bbK(\bx) := \frac{4 \mu}{3\mu + \lambda} \frac{1}{\omega_d} \frac{1}{(|\bx|^2+1)^{\frac{d+1}{2}}} \bbI_{d+1} + \frac{ \mu + \lambda}{3\mu + \lambda} \frac{2(d+1)}{\omega_d} \frac{(\bx,1) \otimes (\bx,1)}{(|\bx|^2+1)^{\frac{d+3}{2}}}$$
is shown to be the Poisson kernel for the Lam\'e system in the upper half space, see \cite[Lemma 5.1]{Martell-M32014}. 
In addition the scaled kernel  $\bbK_t(\bx) := t^{-d} \bbK(\bx/t)$ is shown be an approximation to the identity. As a consequence, to prove the above lemma  it suffices to note  
that for every $\bx \in \bbR^d$, $t > 0$,
\begin{equation}\label{eq-RelationBetweenKandPinProof}
\frac{\mu + \lambda}{3 \mu + \lambda} \mathbb{P}_t(\bx) = \bbK_t(\bx) - \frac{2 \mu}{3 \mu + \lambda} p_t(\bx) \bbI_{d+1}, 
\end{equation}
where we recall that $p_t(\bx)$ the Poisson kernel of the Laplacian in the upper half space.  For each $t> 0$, 
we can now integrate both sides of \eqref{eq-RelationBetweenKandPinProof} to get
\[
\frac{\mu + \lambda}{3 \mu + \lambda} \int_{\bbR^d}{ \mathbb{P}_t(\bx)}{\, \mathrm{d}\bx} = \int_{\bbR^d}{\bbK_t(\bx)}{\, \mathrm{d}\bx} - \frac{2 \mu}{3 \mu + \lambda} \left( \int_{\bbR^d}{p_t(\bx)}{\, \mathrm{d}\bx} \, \right) \bbI_{d+1}
\] 
which yields 
$
\int_{\bbR^d}{\mathbb{P}_t(\bx)}{\mathrm{d}\bx} = \frac{3 \mu + \lambda}{\mu + \lambda} \left( 1 - \frac{2 \mu}{3 \mu + \lambda} \right) \bbI_{d+1} = \bbI_{d+1}\,.
$

The smoothness and the convergences of the matrix convolutions also follow from the same results for $\bbK$ and  $p_t$. It can also be easily verified using Minkowski's inequality and the Lebesgue dominated convergence theorem as follows: 
\[
\begin{split}
\Vnorm{\mathbb{P}_t \, \ast \, \bF - \bF}_{L^p(\bbR^d)}^p &= \int_{\bbR^d}{\left| \int_{\bbR^d}{\mathbb{P}_t(\by) (\bF(\bx-\by) - \bF(\bx))}{\, \mathrm{d} \by} \right|^p}{\, \mathrm{d}\bx} \\
&\leq C(d)\int_{\bbR^d}{|\mathbb{P}_t(\by)| \Vnorm{\tau_{\by} \bF - \bF}_{L^p(\bbR^d)}^p }{\, \mathrm{d}\by}\\
&= C(d)\int_{\bbR^d}{|\mathbb{P}(\by)| \Vnorm{\tau_{t \by} \bF - \bF}_{L^p(\bbR^d)}^p }{\, \mathrm{d}\by} \to 0\,.
\end{split}
\]
where $\tau_{\by} \bF = \bF(\bx-\by)$. 
The integrand in the last term is bounded by the  $L^{1}$ function  $C \, |\mathbb{P}(\by)| \,  \|\bF\|_{L^{p}}^{p}$, and for each $\by$ the integrand converges to zero by  continuity of translations in $L^p$. 
\end{proof}
\begin{remark}
A connection between $\bbP_t$ and the semi-norm $|\cdot|_{\mathcal{X}^{s}_{p}}$ is obtained through the following important relation that we use below. 
For any $\bz$, $\bx\in \bbR^{d}$, we have  
\[
\bbP_t(\bx) \begin{bmatrix} \bz\\0\end{bmatrix}  =  \overline{\bP}(\bx, t)  \left(\bz \cdot \frac{\bx}{|\bx|}\right) \,,\qquad \bz \in \bbR^d\,, 
\]
where the vector function $\overline{\bP}(\bx, t)$ is given by 
$
\overline{\bP}(\bx, t):= \frac{2(d+1)}{\omega_d} \frac{t|\bdx|}{(|\bx|^2 + t^2)^{\frac{d+3}{2}}} \begin{bmatrix} \bx\\t\end{bmatrix}
$.
As a consequence of this and the approximation to the identity result, we see that if $\bF = (\bff, 0)$, then 
\[
\begin{split}
\, \qquad \bU(\bx,t)  &= \bF(\bdx)  + \int_{\mathbb{R}^{d}} \bbP_t (\bdy)(\bF(\bdx + \bdy) - \bF(\bdx)) \, \mathrm{d}\bdy =  \bF(\bdx)  + \int_{\mathbb{R}^{d}} \overline{\bP}(\bdy, t)(\bff(\bdx + \bdy) - \bff(\bdx))\cdot {\bdy\over |\bdy|} \, \mathrm{d}\bdy\,.
 \end{split}
\]
\end{remark}
The matrix Poisson kernel $\mathbb{P}_t$ also satisfies a semigroup property as documented in the next lemma.  The following arguments rely centrally on an explicit formula for the Fourier transform of $\mathbb{P}_t$.
\begin{lemma}
 For each $t>0$, the Fourier transform of $\mathbb{P}_{t}(\bx)$ is given by 
\begin{equation}\label{eq-FourierTransformOfPoissonTypeKernel}
\cF_{\bx}({\mathbb{P}_t})(\bfxi) := \e^{- 2 \pi |\bfxi| t} \left( \bbI_{d+1} + (2 \pi |\bfxi| t)
	\begin{bmatrix}
		- \frac{\bfxi \otimes \bfxi}{|\bfxi|^2} & -\imath \frac{\bfxi}{|\bfxi|} \\
		-\imath \frac{\bfxi}{|\bfxi|} & 1 \\
	\end{bmatrix}
\right)\,.
\end{equation}
As a consequence $\mathbb{P}_t$ satisfies the semigroup property: for every  $t_1, t_2 > 0\,$ 
\[
\mathbb{P}_{t_1} \, \ast \, \mathbb{P}_{t_2} = \mathbb{P}_{t_1 + t_2}\,, 
\]
where the convolution is understood in the sense of matrix multiplication. 
\end{lemma}

\begin{proof} 
To preserve the flow of the presentation in this section, the Fourier transform of $\mathbb{P}_t$ is computed in the appendix.  To prove the semigroup property of $\mathbb{P}_t $ we use the property of convolution and the explicit Fourier transform formula given in \eqref{eq-FourierTransformOfPoissonTypeKernel}.   We carry out this calculation via matrix multiplication. For any positive $t_1$, $t_2$ we have that 
\[
\begin{split}
	\cF_{\bx} \left( \mathbb{P}_{t_1} \, \ast \, \mathbb{P}_{t_2} \right)(\bfxi) &= \cF_{\bx}(\mathbb{P}_{t_1})(\bfxi) \, . \, \cF_{\bx}(\mathbb{P}_{t_2})(\bfxi) \\ 
	&= \e^{-2 \pi |\bfxi| (t_1 + t_2)} \left( \bbI_{d+1} + (2 \pi |\bfxi| t_1) 
		\begin{bmatrix}
		- \frac{\bfxi \otimes \bfxi}{|\bfxi|^2} & -\imath \frac{\bfxi}{|\bfxi|} \\
		-\imath \frac{\bfxi}{|\bfxi|} & 1 \\
		\end{bmatrix} + (2 \pi |\bfxi| t_2) 
		\begin{bmatrix}
		- \frac{\bfxi \otimes \bfxi}{|\bfxi|^2} & -\imath \frac{\bfxi}{|\bfxi|} \\
		-\imath \frac{\bfxi}{|\bfxi|} & 1 \\
		\end{bmatrix} \right) \\
		&=  \cF_{\bx}(\mathbb{P}_{t_1 + t_2})(\bfxi)\,, 
		\end{split}
		\]
where in the second equality we have multiplied the matrix of Fourier symbols and have also used the fact that  
$\begin{bmatrix}
	- \frac{\bfxi \otimes \bfxi}{|\bfxi|^2} & -\imath \frac{\bfxi}{|\bfxi|} \\
		-\imath \frac{\bfxi}{|\bfxi|} & 1 \\
\end{bmatrix}^2 = {\bfs 0}\,$, the zero matrix,  which can be verified easily by computation. 
\end{proof}
The next lemma summarizes integrability properties of the first derivatives of $\bbP_{t}$ that we will be using later. The proof is purely computational and can be done following similar calculations for the Poisson kernel. We omit it here.
\begin{lemma}\label{integrable-PType}
For each $j$, $k$, and $\ell \in \{ 1, \ldots, d, d+1 \}$ and for every $t > 0$ we have that $\p_t \mathfrak{p}^{jk}_t(\bx) \in L^1(\bbR^d)$ and $\p_{x_\ell} \mathfrak{p}^{jk}_t(\bx) \in L^1(\bbR^d)$.  In addition we have the following pointwise estimates: There exists a constant $c = c(d) > 0$ such that
\[
|\p_t \mathfrak{p}^{jk}_t(\bx)| \leq c \,|\bx|^{-d-1}\,, \qquad |\p_t \mathfrak{p}^{jk}_t(\bx)| \leq c\, t^{-d-1}\,, \qquad \bx \in \bbR^{d}\,, \quad t > 0\,, 
\]
for any $j, k = 1, 2, \dots d+1$.
\end{lemma}

\subsection{Norm equivalence of Poisson integrals}
We begin first by establishing relations between the Poisson integrals obtained from $p_t$ and $\mathbb{P}_{t}$. 
\begin{lemma}\label{prop-identity}
Suppose that $\bF = (\bff, 0) \in \cS(\bbR^d;\mathbb{R}^{d+1})$. Then for every $t > 0$ both Poisson integrals $\bu(\bx,t) = p_t \ast \bF(\bx) $ and $\bU(\bx,t)  = \mathbb{P}_{t}\ast \bF(\bx)$ are in $\cS(\bbR^d;\mathbb{R}^{d+1})$. Moreover we have the following relations between ${\bf u}$ and ${\bf U}$. 
\begin{itemize}
\item For any $j=1, \dots, d,$ we have  
\begin{equation}\label{eq-IdentityForComparisonOfuandU}
{U_j}(\bx,t)  = {u_j}(\bx,t) + {R_j(U_{d+1})}(\bx,t)\,,
\qquad  \bx\in\mathbb{R}^{d}\,, \quad t > 0\,.
\end{equation}
\item 
\begin{equation}\label{eq-IdentityForComparisonOfDerivativesOfUd+1}
\p_t U_{d+1}(\bx,t) = -\mathrm{div}_{\bx}{\bu}(\bx,t) - \sum_{j=1}^d R_j  (\p_{x_j} U_{d+1})(\bx,t)\,, \qquad  \bx\in\mathbb{R}^{d}\,, \quad t > 0\,.
\end{equation}
\item For any $j=1, \dots, d,$ we have   \begin{equation}\label{eq-IdentityForComparisonOfDerivativesOfUd+1-IV}
\p_{x_j}U_{d+1}(\bx,t) = R_j \left( \p_t U_{d+1} + \sum_{\ell=1}^d R_{\ell} \left( \p_t u_{\ell} \right) \right) (\bx,t)\,, \qquad  \bx\in\mathbb{R}^{d}\,, \quad t > 0\,,
\end{equation}
\end{itemize}
where $R_j$ is the $j^{th}$ Riesz transform.  
\end{lemma}
\begin{proof}
We prove first the identity \eqref{eq-IdentityForComparisonOfuandU}. For a fixed $t > 0$, since all the functions  involved are in $\mathcal{S}(\bbR^{d})$, it suffices to check that  the Fourier transform of the right-hand side agrees with that of the left-hand side in \eqref{eq-IdentityForComparisonOfuandU}. From the definition of $\bU$, we see that for any $t > 0$,  $\cF_{\bx}(\bU (\cdot, t))(\bfxi) = \cF(\mathbb{P}_{t})(\bfxi){\cF(\bF)}(\bfxi)$. Now from the particular form of $\bF$ and using the explicit formula \eqref{eq-FourierTransformOfPoissonTypeKernel} for the Fourier transform of $\bbP_t$ we see that for any $j=1, \dots, d$ we have 
\begin{equation}\label{jth}
\cF_{\bx}({U_j})(\bfxi, t) = \e^{- 2 \pi |\bfxi| t}  \cF({f_j})(\bfxi) + \e^{- 2 \pi |\bfxi| t} (2 \pi |\bfxi| t) \left( - \frac{\xi_j}{|\bfxi|} \right) \left( \frac{\bfxi}{|\bfxi|} \cdot \cF({\bff})(\bfxi) \right)  \\
\end{equation}
and that after simplification $\cF_{\bx}({U_{d+1}})(\bfxi, t) = -\imath2 \pi t \,\e^{- 2 \pi |\bfxi| t} \bfxi \cdot \cF({\bff})(\bfxi)  $. 
To complete  the proof of the identity  \eqref{eq-IdentityForComparisonOfuandU} we notice that the first term  in \eqref{jth} is precisely $\cF_{\bx}(u_j)(\bfxi, t)$, whereas the second term can be rewritten to obtain 
\[
\begin{split}
\e^{- 2 \pi |\bfxi| t} (2 \pi |\bfxi| t) \left( - \frac{\xi_j}{|\bfxi|} \right) \left( \frac{\bfxi}{|\bfxi|} \cdot \cF({\bff})(\bfxi) \right) &= \left( -\imath\frac{\xi_j}{|\bfxi|} \right) \left( -\imath(2 \pi t) \e^{- 2 \pi |\bfxi| t} \left(  \bfxi \cdot \cF({\bff})(\bfxi) \right) \right) \\
&= \cF_{\bx}\left({R_j(U_{d+1})}\right)(\bfxi,t)\,.
\end{split}
\]
Let us proceed to show \eqref{eq-IdentityForComparisonOfDerivativesOfUd+1}. Using a direct calculation and some rearrangement we get 
\begin{equation} \label{tder-Udplus1}
\begin{split}
\p_t \cF_{\bx}({U_{d+1}})(\bfxi,t) &= (-\imath2 \pi ) \e^{-2 \pi |\bfxi| t} \bfxi \cdot \cF({\bff})(\bfxi)  +\imath(4 \pi^2 |\bfxi| t) \e^{-2 \pi |\bfxi| t} \bfxi \cdot \cF({\bff})(\bfxi)\,.
\end{split}
\end{equation}
The identity follows easily once we realize that the first term of \eqref{tder-Udplus1} can be rewritten as 
\[(-\imath2 \pi ) \e^{-2 \pi |\bfxi| t} \bfxi \cdot \cF({\bff})(\bfxi) 
= -2 \pi \imath \bfxi \cdot \e^{-2 \pi |\bfxi'| t}\cF({\bff})(\bfxi) = -\cF_{\bx}\left(\text{div}_{\bx}{\bu}(\cdot, t)\right)(\bfxi)\,,   \]
while the second term in \eqref{tder-Udplus1} can also be rewritten as 
\[
\begin{split}
\imath(4 \pi^2 |\bfxi| t) \e^{-2 \pi |\bfxi| t} \bfxi \cdot \cF({\bff})(\bfxi) &=(2 \pi |\bfxi|)(\imath2 \pi t) \e^{-2 \pi \bfxi t} \left(  \bfxi \cdot \cF({\bff})(\bfxi) \right)\\
 &=  -\sum_{j=1}^d  \left( -\imath \frac{\xi_j}{|\bfxi|} \right) 2 \pi \imath \xi_j  (-\imath2 \pi t) \e^{-2 \pi |\bfxi| t} \left(  \bfxi \cdot \cF({\bff})(\bfxi) \right)\\
&=  -\cF\left( \sum_{j=1}^{d} R_{j}\left(\partial_{j}U_{d+1}\right)\right) (\bfxi, t)\,.
\end{split}
\]
Next we prove the identity \eqref{eq-IdentityForComparisonOfDerivativesOfUd+1-IV}. Again, by a direct calculation
\begin{equation}\label{eq-IdentityForComparisonOfDerivativesOfUd+1-II}
\begin{split}
\cF({\p_{x_j} U_{d+1}})(\bfxi,t) &= (\imath2 \pi  \xi_j)(-\imath2 \pi t) \e^{-2 \pi |\bfxi| t} \bfxi\cdot \cF({\bff})(\bfxi) = \left( -\imath \frac{\xi_j}{|\bfxi|} \right) (4 \pi^2 |\bfxi| t) \e^{-2 \pi |\bfxi| t} \left( \imath \bfxi \cdot \cF({\bff})(\bfxi) \right)\,.
\end{split}
\end{equation}
We need to connect the last expression in \eqref{eq-IdentityForComparisonOfDerivativesOfUd+1-II} with  $\p_t U_{d+1}$. To do so, we observe from \eqref{tder-Udplus1} that 
\begin{equation}\label{eq-IdentityForComparisonOfDerivativesOfUd+1-III}
(4 \pi^2 |\bfxi|^2 t) \e^{-2 \pi |\bfxi| t} \left( \imath \frac{\bfxi}{|\bfxi|} \cdot \cF({\bff})(\bfxi) \right) =  \cF_{\bx}(\p_{t} U_{d+1})(\bfxi,t) + \cF \left( \sum_{k=1}^d R_k \left[ \p_t u_k(\cdot,t) \right] \right)(\bfxi)\,,
\end{equation}
where the last term is a rewriting of the expression $(-\imath2 \pi ) \e^{-2 \pi |\bfxi| t} \bfxi \cdot \cF({\bff})(\bfxi)$ in \eqref{tder-Udplus1}. 
Substituting this into \eqref{eq-IdentityForComparisonOfDerivativesOfUd+1-II}  we get the desired result. 
\end{proof}
\begin{proposition}\label{lpcomparisonDeruwithDerU}
Let $1 < p < \infty$. Let $\bff \in L^p(\bbR^d;\bbR^d)$. Set $\bF = (\bff, 0) \in L^{p}(\bbR^d;\mathbb{R}^{d+1})$, $\bu(\bx,t) = p_t \ast \bF(\bx) $ and set $\bU(\bx,t)  = \mathbb{P}_{t}\ast \bF(\bx)$. Then there exists a positive constant $C = C(d, p)$ such that for any $t > 0$ we have 
\begin{equation}\label{eq-KornsForPoissonIntegrals-dtj}
\Vnorm{\p_t \bu(\cdot,t)}_{L^p(\bbR^d)} \leq C \Vnorm{\p_t \bU(\cdot,t)}_{L^p(\bbR^d)} 
\end{equation}
 and for each $k=1, \dots, d$ we have 
\begin{equation}\label{eq-KornsForPoissonIntegrals-Dxj}
 \Vnorm{\p_{x_k} \bu(\cdot,t)}_{L^p(\bbR^d)} \leq C  \Vnorm{\p_{x_k} \bU(\cdot,t)}_{L^p(\bbR^d)}\,.
\end{equation}
\end{proposition}
\begin{proof}
We prove both inequalities for $\bff \in  \cS(\bbR^d;\mathbb{R}^{d})$ and then the general case follows by density. Both inequalities \eqref{eq-KornsForPoissonIntegrals-dtj} and \eqref{eq-KornsForPoissonIntegrals-Dxj} follow from identity  \eqref{eq-IdentityForComparisonOfuandU} in Proposition \ref{prop-identity}. Indeed, for $j=1, \dots, d$, we can differentiate the equation \eqref{eq-IdentityForComparisonOfuandU} in $t$ to obtain that for any $\bx\in \mathbb{R}^{d}$ and $t > 0$ 
\[\p_t u_j(\bx,t) = \p_t U_j(\bx,t) - R_j \big( \p_t U_{d+1} \big) (\bx,t)\,. \]
In a similar fashion, if we differentiate the equation \eqref{eq-IdentityForComparisonOfuandU}  in $x_k$ we obtain that 
\[
\p_{x_{k}} u_j(\bx,t) = \p_{x_{k}} U_j(\bx,t) - R_j \big( \p_{x_{k}} U_{d+1} \big) (\bx,t)\,.
\]
Both inequalities \eqref{eq-KornsForPoissonIntegrals-dtj} and \eqref{eq-KornsForPoissonIntegrals-Dxj} now follow by taking the $L^{p}$ norm on both sides of the above two equations and summing over $j=1, \dots, d$.  
Note that we have used both the fact that the Riesz transforms commute with differential operators and that the Riesz transforms are $L^p$ bounded.

For general $\bff \in L^p(\bbR^d;\bbR^d)$ take a sequence $(\bff_n) \subset \cS(\bbR^d;\bbR^d)$ converging to $\bff$ in $L^p(\bbR^d;\bbR^d)$. Set $\bF_n = (\bff_n, 0)$. Then apply  Lemma \ref{thm-PropertiesOfPoissonTypeKernelP},  Lemma \ref{integrable-PType} and Young's inequality to see that  $\p_t \big[ \mathbb{P}_t \, \ast \, \bF_n \big]$ converges to $\p_t \big[ \mathbb{P}_t \, \ast \, \bF \big]$ in $L^p(\bbR^d;\bbR^d)$ and that $\nabla_\bx \big[ \mathbb{P}_t \, \ast \, \bF_n \big]$ converges to $\nabla_\bx \big[ \mathbb{P}_t \, \ast \, \bF \big]$ in $L^p(\bbR^d;\bbR^d)$. 
\end{proof} 

\section{A characterization of fractional Sobolev spaces} \label{characterization}
\subsection{Equivalence of spaces}
In this subsection we prove one of the main results of this paper, which is the equivalence of the spaces $\mathcal{X}^{s}_{p}(\mathbb{R}^{d})$ and  $W^{s, p}(\bbR^d;\mathbb{R}^{d})$. We paraphrase it in the following theorem. 

\begin{theorem}\label{3inequalities}
Let $s\in (0, 1)$ and $1 <p < \infty$. Then $\mathcal{X}^{s}_{p}(\mathbb{R}^{d}) = W^{s, p}(\bbR^d;\mathbb{R}^{d})$. Moreover, there are constants $C_{i}$, $i=1, 2,3$ all depending only on $s, p, $ and $d$ such that for any $\bff = (f_1,f_2, \ldots, f_d) \in L^p(\bbR^d)$
\[
\begin{split}
[{\bf f}]_{W^{s, p}} & \stackrel{\bf (EQ_{1})}{\leq} \quad C_{1} \int_{0}^{\infty}{ t^{p(1-s)} \Vnorm{\p_t \bu(\cdot,t)}^{p}_{L^p(\bbR^d)} \frac{1}{t}}{dt}\\
&  \stackrel{\bf (EQ_{2})}{\leq} \quad C_{2} \int_{0}^{\infty}{ t^{p(1-s)} \Vnorm{\p_t \bU(\cdot,t)}^{p}_{L^p(\bbR^d)} \frac{1}{t}}{dt}\\
& \stackrel{\bf (EQ_{3})}{\leq}\quad C_{3} [{\bf f}]_{\mathcal{X}^{s}_{p}} \,,
\end{split}
\]
where ${\bf F}  = (\bff, 0)$, ${\bf u}(\bdx, t) = p_{t}\ast {\bff}(\bdx)$, and ${\bf U}(\bdx, t) = \mathbb{P}_{t}\ast {\bf F}(\bdx)$.
\end{theorem}

\begin{proof}
In the above (${\bf EQ_{1}}$) is in \eqref{semi-norm-EQUIV} proved in Proposition \ref{thm-Prop7InStein} and  (${\bf EQ_{2}}$) follows from the pointwise-in-$t$ estimate proved in Proposition \ref{lpcomparisonDeruwithDerU}.  What remains is the proof of the inequality (${\bf EQ_{3}}$). We prove it as follows.  
Recalling that $\int_{\mathbb{R}^{d}} \mathbb{P}_t (\bdx) \, \mathrm{d}\bx = \mathbb{I}_{d+1}$, we have 
\begin{equation}
\p_t \bU(\bx,t) = \intdm{\bbR^d}{\p_t \mathbb{P}_t(\by) \bF(\bx-\by)}{\by} = \int_{\bbR^d}{\p_t \mathbb{P}_t(\by)\left( \bF(\bx-\by) - \bF(\bx) \right)}{\, \mathrm{d}\by}\,.
\end{equation}
To reveal the connection with the integrand in the semi-norm  $ [{\bf f}]_{\mathcal{X}^{s}_{p}}$ we compute the derivative $\p_t \mathbb{P}_t(\by)$ in the above convolution directly. For $j = 1, \dots, d$ the $j^{th}$ term is given by 
\begin{equation}\label{reveal1}
\begin{split}
	\p_t U_j(\bx,t) &= \frac{2(d+1)}{\omega_d} \sum_{k=1}^d \int_{\bbR^d}{ \left( \frac{y_j}{(|\by|^2+t^2)^{\frac{d+3}{2}}} - \frac{(d+3)y_jt^2}{(|\by|^2+t^2)^{\frac{d+5}{2}}} \right) \Big( y_k \cdot ( f_k(\bx-\by) - f_k(\bx)) \Big)}{\, \mathrm{d}\by} \\
&= \frac{2(d+1)}{\omega_d} \int_{\bbR^d}{ \left( \frac{y_j |\by|}{(|\by|^2+t^2)^{\frac{d+3}{2}}} - \frac{(d+3)y_j \, |\by| \, t^2}{(|\by|^2+t^2)^{\frac{d+5}{2}}} \right) \left( (\bff(\bx+\by) - \bff(\bx)) \cdot \frac{\by}{|\by|} \right)}{\mathrm{d}\by}\,.
\end{split}
\end{equation}
A similar computation also shows that 
\begin{equation}\label{reveal2}
\begin{split}
\p_t U_{d+1}(\bx,t) &= \frac{2(d+1)}{\omega_d} \sum_{k=1}^d \int_{\bbR^d}{ \left( \frac{2t}{(|\by|^2+t^2)^{\frac{d+3}{2}}} - \frac{(d+3)t^3}{(|\by|^2+t^2)^{\frac{d+5}{2}}} \right) \Big( y_k \cdot ( f_k(\bx-\by) - f_k(\bx)) \Big)}{\, \mathrm{d}\by} \\
&= - \frac{2(d+1)}{\omega_d} \int_{\bbR^d}{ \left( \frac{2t |\by|}{(|\by|^2+t^2)^{\frac{d+3}{2}}} - \frac{(d+3)|\by| t^3}{(|\by|^2+t^2)^{\frac{d+5}{2}}} \right) \left( (\bff(\bx+\by) - \bff(\bx)) \cdot \frac{\by}{|\by|} \right)}{\, \mathrm{d}\by}\,. 
\end{split}
\end{equation}
Notice that the expressions inside the integrals in \eqref{reveal1} and \eqref{reveal2} are linear combinations of the $\p_t \mathfrak{p}_{t}^{jk}$ after factoring the unit vector $ {\by\over |\by|}$. As a result, these expressions enjoy the same pointwise estimates as  $\p_t \mathfrak{p}_{t}^{jk}$ stated in the Lemma \ref{integrable-PType}. That is, the expressions are majorized by $t^{-d-1}$ as well as by $|\by|^{-d-1}$.  We  will make use of these pointwise estimates below. 

By splitting the convolution integrals in \eqref{reveal1} and \eqref{reveal2} into an integral over $B_{t}({\bfs 0})$ and $\complement B_{t}(\bfs{0})$, the complement of $B_{t}(\bfs{0})$, we obtain that for any $t> 0$ and $\bx \in \mathbb{R}^d$
\[
\begin{split}
	\left| \p_t \bU(\bx,t) \right| &\leq
	\frac{C}{t^{d+1}} \int_{|\by| \leq t}{ \left| (\bff(\bx+\by) - \bff(\bx)) \cdot \frac{\by}{|\by|} \right|}{\, \mathrm{d}\by} + C \int_{|\by| > t}{ \left| (\bff(\bx+\by) - \bff(\bx)) \cdot \frac{\by}{|\by|} \right| \frac{1}{|\by|^{d+1}}}{\, \mathrm{d}\by}\,. 
\end{split}
\]
Now, using Minkowski's integral inequality we obtain that 
\begingroup\makeatletter\def\f@size{8.5}\check@mathfonts
\begin{equation}\label{two-pieces}
\begin{split}
	\Vnorm{\p_t \bU(\cdot,t)}_{L^p(\bbR^d)} \leq \frac{C}{t^{d+1}} \int_{|\by| \leq t}{\Vnorm{ (\bff(\cdot+\by) - \bff(\cdot)) \cdot \frac{\by}{|\by|}}_{L^p(\bbR^d)} }{\, \mathrm{d}\by} + C \int_{|\by|>t}\frac{{\Vnorm{ (\bff(\cdot+\by) - \bff(\cdot)) \cdot \frac{\by}{|\by|}}_{L^p(\bbR^d)} }}{|\by|^{d+1}} \, \mathrm{d}\by\,. 
\end{split}
\end{equation}
\endgroup
The remaining part of the argument that estimates the right-hand side of the above inequality by the semi-norm $ [{\bf f}]_{\mathcal{X}^{s}_{p}}$ follows that of \cite[Page 152]{Stein} where it was done for classical Besov spaces. We repeat it here for clarity. 
Changing to polar coordinates, write $\by = r \bw \in \bbR^d$, with $r = |\by|$ and $\bw \in \bbS^{d-1}$. Define
\[
\Psi(r) := \int_{\bbS^{d-1}}{\Vnorm{ (\bff(\cdot+r\bw) - \bff(\cdot)) \cdot {\bw}}_{L^p(\bbR^d)}}{\, \mathrm{d}\sigma(\bw)}\,.
\]
Then the last inequality in \eqref{two-pieces}  can be rewritten in terms of $\Psi(r)$ to obtain 
\[
\Vnorm{\p_t \bU(\cdot,t)}_{L^p(\bbR^d)} \leq \frac{C}{t^{d+1}} \int_{0}^{t}{\Psi(r) r^{d-1}}{\, \mathrm{d}r} + C \int_{t}^{\infty}{\Psi(r) r^{-2}}{\, \mathrm{d}r}\,.
\]
Multiply both sides by $t^{1-s}$ and estimate the norm in $L^{p}((0, \infty); t^{-1} \, \mathrm{d}t )$ on both sides to obtain, using Hardy's inequalities \cite[Appendix~A]{Stein},   that 
\begin{equation}
\begin{split}
	\left( \int_{0}^{\infty}{\left( t^{1-s} \Vnorm{\p_t \bU(\cdot,t)}_{L^p(\bbR^d)} \right)^p \frac{1}{t} }{\, \mathrm{d}t} \right)^{1/p} \leq {C\over (1-s)(s + d)}\left( \int_{0}^{\infty}{ \left( \Psi(r) r^{-s} \right)^p \frac{1}{r}}{\mathrm{d}r} \right)^{1/p}\,.
\end{split}
\end{equation}
By  H\"older's inequality we have $\Psi(r)^p \leq C \int_{\bbS^{d-1}}{\Vnorm{ (\bff(\cdot+r\bw) - \bff(\cdot)) \cdot {\bw}}_{L^p(\bbR^d)}^p}{\, \mathrm{d}\sigma(\bw)}$  and so we have 
\begin{equation}
\begin{split}
	&\left( \int_{0}^{\infty}{\left( t^{1-s} \Vnorm{\p_t \bU(\cdot,t)}_{L^p(\bbR^d)} \right)^p \frac{1}{t} }{\, \mathrm{d}t} \right)^{1/p} \\
	&\quad\quad\leq {C\over (1-s)(s + d)} \left( \int_{0}^{\infty}\int_{\bbS^{d-1}}{\Vnorm{ (\bff(\cdot+r\bw) - \bff(\cdot)) \cdot {\bw}}_{L^p(\bbR^d)}^p} r^{-sp-1} \, \mathrm{d}\sigma(w) {\, \mathrm{d}r} \right)^{1/p} \\
	&\quad\quad\leq {C\over (1-s)(s + d)} \left( \int_{0}^{\infty} \int_{\bbS^{d-1}} { {\Vnorm{ (\bff(\cdot+r\bw) - \bff(\cdot)) \cdot {\bw}}_{L^p(\bbR^d)}^p} \over r^{d+sp} } r^{d-1} {\, \mathrm{d}\sigma(w)} {\, \mathrm{d}r} \right)^{1/p} \\
	&\quad\quad\leq {C\over (1-s)(s + d)} \left( \int_{\bbR^d}{\frac{{\Vnorm{ (\bff(\cdot+\by) - \bff(\cdot)) \cdot {{\by}\over |\by|}}_{L^p(\bbR^d)}^p}}{|\by|^{d+sp}}}{\, \mathrm{d}\by} \right)^{1/p}\\
	 &\quad\quad= C  [\bff]_{\mathcal{X}^{s}_{p}(\bbR^d)}\,.
\end{split}
\end{equation}
where the last $C$ depends only only  on $s, p,$ and $d$. 
This completes  the proof of the theorem. 
\end{proof}
\subsection{Applications}
One may now use the equivalence of spaces we have established to obtain inequalities that are important in application. The simplest of all is the fractional Poincar\'e-Korn inequality which we will need in the next section. 
\begin{corollary}\label{Cor:P-korn}
Let $s \in (0,1)$ and $p \in (1,\infty)$. Let $B\subset \bbR^d$ be a ball. Then there exists a constant $C = C(p,s,d,B) >0$ such that 
\[
	\Vnorm{\bff}_{L^{p}(\bbR^d)}^p \leq C \int_{\bbR^d}\int_{\bbR^d}{\frac{\left|(\bff(\bx)-\bff(\by)) \cdot \frac{\bx-\by}{|\bx-\by|} \right|^p}{|\bx-\by|^{d+sp}}}{\, \mathrm{d}\bx}{\, \mathrm{d}\by}\,. 
\]
for any $\bff \in {W}^{s,p}(\bbR^d;\mathbb{R}^{d})$ such that $\text{Supp}(\bff)\subset B$. 
\end{corollary}
Another corollary of Theorem \ref{3inequalities} is a fractional Sobolev embedding \cite{Hitchhiker} that uses the seminorm $\mathcal{X}^{s}_{p}(\bbR^d)$. 
\begin{corollary}\label{Cor-SE}
Let $s \in (0,1)$ and $p \in (1,\infty)$ such that $sp < d$. Then there exists a constant $C = C(d,p,s)$ such that for any measurable and compactly supported vector field $\bff:\bbR^d\to \bbR^{d}$ we have 
\begin{equation*}
	\Vnorm{\bff}_{L^{p^{\ast_{s}}}(\bbR^d)}^p \leq C \int_{\bbR^d}\int_{\bbR^d}{\frac{\left|(\bff(\bx)-\bff(\by)) \cdot \frac{\bx-\by}{|\bx-\by|} \right|^p}{|\bx-\by|^{d+sp}}}{\, \mathrm{d}\bx}{\, \mathrm{d}\by}\,. 
\end{equation*}
where $p^{\ast_s} = \frac{dp}{d-sp}$.  As a consequence the space $\cX^{s}_{p}(\bbR^d)$ is continuously embedded in $L^{q}(\bbR^d;\bbR^d)$, for any $q\in [p, p^{\ast_s}]$. 
\end{corollary}
We will also use Theorem \ref{3inequalities} to prove fractional Korn-type inequalities for functions defined on the half space $\mathbb{R}^{d}_{+}$. The argument to prove such results is standard. We first extend vector fields to be defined over $\mathbb{R}^{d}$ such that the norm of the extended vector field is controlled by the original one.  Such an extension theorem is recently proved in \cite{Mengesha-Hardy}, which we state below. \begin{theorem} [{\bf Extension operator}]\label{thm:extension}
Let $d\geq 1$, $p\in [1,\infty)$ and $0<s<1$ and $ps\neq 1$. There exists an extension operator $$E : \mathring{\mathcal{X}}^{s}_{p}(\mathbb{R}^{d}_{+}) \to \mathcal{X}_{p}^{s}(\mathbb{R}^{d}) $$ and a positive constant $C = C(p,d,s)$ such that for any  ${\bf f}\in \mathring{\mathcal{X}}_{p}^{s}(\mathbb{R}^{d}_{+})$, and $\tilde{{\bf f}} = E{\bf f}$ we have that $\tilde{{\bf f}}  = {\bf f}$ a.e. in $\mathbb{R}^{d}_{+}$, $\tilde{{\bf f}} \in \mathcal{X}^{s}_{p}(\mathbb{R}^{d})$ and
\[
\begin{split}
&|\tilde{{\bf f}} |_{\mathcal{X}^{s}_{p}(\mathbb{R}^{d})} \leq C |{\bf u}|_{\mathcal{X}^{s}_{p}(\mathbb{R}^{d}_{+})}. 
\end{split}
\]
\end{theorem}
The theorem is proved in \cite{Mengesha-Hardy}. We emphasize that the proof of the above extension theorem is nontrivial as the commonly used reflection across the hyperplane $x_{d} = 0$ would not preserve the semi-norm $|\cdot|_{\mathcal{X}^{s}_{p}}$. Extending by zero is also not appropriate, since it is not clear how to control the norm of the extended function. Rather we use an extension operator that has been used by J. A. Nitsche in \cite{Nitsche} in his simple proof of Korn's second inequality.  In showing the boundedness of the extension operator with respect to the semi-norm $|\cdot|_{\mathcal{X}^{s}_{p}}$ 
 we need to first establish the fractional Hardy-type inequality.   See \cite{Mengesha-Hardy} for more details.  With extension at hand the proof of the result below is standard. 
\begin{proposition}
For $s \in (0,1)$, $1 < p < \infty$, $sp \neq 1$, there exists a constant $C = C(d,p,s) > 0$ such that for any ${\bf f}\in \mathring{\mathcal{X}}^{s}_{p}(\mathbb{R}^{d}_{+})$ it holds that 
\begin{equation}
 \int_{\bbR^d_{+}}\int_{\bbR^d_{+}}{\frac{\left|\bff(\bx)-\bff(\by) \right|^p}{|\bx-\by|^{d+sp}}}{\, \mathrm{d}\bx}{\, \mathrm{d}\by}\,  \leq C \int_{\bbR^d_{+}}\int_{\bbR^d_{+}}{\frac{\left|(\bff(\bx)-\bff(\by)) \cdot \frac{\bx-\by}{|\bx-\by|} \right|^p}{|\bx-\by|^{d+sp}}}{\, \mathrm{d}\bx}{\, \mathrm{d}\by}\,. 
\end{equation}
\end{proposition}
In future work we hope to report on  the natural next step of establishing  the equivalence of the $\mathcal{X}^{s}_{p}(\Omega)$ with $W^{s, p}(\Omega;\mathbb{R}^{d})$ defined over domains with sufficiently regular boundary. 
\section{Self improving properties for a coupled system of nonlocal equations}

\subsection{Preliminaries }
Given a ball $B\subset \bbR^{d}$ with radius $r$, $\kappa B$ represents a ball with the same center but with radius $\kappa r$.
Note that for a given ${\bf u} \in \mathcal{X}^{s}_{p}(2B)$ and $\eta \in C^{\infty}_c(2B)$ with  $\eta \equiv 1$ in $B$, the function $\eta {\bf u} \in \mathcal{X}^{s}_{p}(\bbR^{d})$. Moreover, 
\[
[\eta {\bf u}]_{ \mathcal{X}^{s}_{p}(\mathbb{R}^{d})} \leq C \|{\bf u}\|_{\mathcal{X}^{s}_{p}(2B)}
\]
for a constant $C$ independent of ${\bf u}$. 

We also recall that for $\bu\in \mathcal{X}_{p}^{s}(\Omega), $ we have that $\bbL^{s}_{p,\Omega}({\bf u})\in [\mathcal{X}_{p}^{s}(\Omega)]^{\ast}$. Indeed,  by definition \eqref{defn-of-operator} and H\"older's inequality we have 
\[
|\langle\bbL^{s}_{p,\Omega}({\bf u}), \phi\rangle| \leq C [{\bf u}]_{\mathcal{X}_p^{s}(\Omega)}[\phi]_{\mathcal{X}_p^{s}(\Omega)}, \quad \text{for all $\phi\in \mathcal{X}_{p}^{s}(\Omega)$}\,. 
\]
More generally, for $t\in (0, 1)$, we define the dual norm of $\bbL^{s}_{p,\Omega}({\bf u})$ in $\mathring{\mathcal{X}}_{p}^{t}(\Omega)$ by 
\[
\| \bbL^{s}_{p,\Omega}({\bf u})\|_{[\mathring{\mathcal{X}}_{p}^{t}(\Omega)]^{\ast}} := \sup_{\phi} |\langle\bbL^{s}_{p,\Omega}({\bf u}), \phi\rangle|\,,
\]
where the supremum is taken over all $\phi\in C^\infty_c(\Omega;\mathbb{R}^{d})$ such that $[\phi]_{\mathcal{X}_{p}^{t}(\bbR^{d}) } \leq 1$. The fractional Laplacian $(-\Delta)^{\beta}{\bf u}$ is defined as
\[
(-\Delta)^{\beta}{\bf u}(\bdx) = c_{s,d} \,\cF^{-1}(|\bfxi|^{2\beta} \cF({\bf u}))\,,
\]
where $c_{s,d}>0$ is a normalizing constant.  
The fractional Laplacian has a quasi-local behavior that has been quantified in the following estimates: Let  $p,q \in [1,\infty]$ and $s \in (0,1)$. 
For any  $\Omega_1$ and $\Omega_2$ be bounded, disjoint open sets such that $\rho := \dist(\Omega_1,\Omega_2) > 0$ and for any $\vphi \in C^{\infty}_c(\bbR^d)$,
$$
\Vnorm{\Dss (\vphi \chi_{\Omega_2})}_{L^p(\Omega_1)} \leq \rho^{-d-2s} |\Omega_1|^{1/p} |\Omega_2|^{1-1/q} \Vnorm{\vphi}_{L^q(\Omega_2)}\,.
$$
This is established in \cite{Schikorra, Schikorra2}.  
We also state the following technical lemma -- in a form needed for our purposes -- that summarizes the action of the fractional Laplacian as a potential. The result in this lemma is embedded in an inequality in \cite{Schikorra}, but we will give the proof here for the sake of completeness.
\begin{lemma}\label{tech-tools}
Let $p,q \in [1,\infty]$ and $s \in (0,1)$. 
Let $\veps \in (0, 1-s)$. Suppose that $B\subset \bbR^{d}$  is a ball, $\vphi \in C^{\infty}_c(4B;\bbR^d)$ with $[\vphi]_{\mathcal{X}^{s+\veps}_{p}(\bbR^d)} \leq 1$, and $\eta \in C^{\infty}_c(6B)$ and $\eta \equiv 1$ on $5B$. 
Then the function $\eta (-\Delta)^{\frac{\veps p}{2}}\vphi\,\in C^{\infty}_c(6B;\mathbb{R}^{d})$, and there exists a constant $C>0,$ independent of $\vphi$ such that 
\begin{equation}\label{eq-TestFxnEstimate1}
[\eta (-\Delta)^{\frac{\veps p}{2}}\vphi]_{\mathcal{X}^{s-\veps(p-1)}_{p}(\bbR^d)} \leq  C [\vphi]_{\mathcal{X}^{s+\veps}_{p}(\bbR^d)}\,.
\end{equation}
\end{lemma}
\begin{proof}
 
Adding and subtracting $\eta(\by) (-\Delta)^{\frac{\veps p}{2}} \vphi(\bx)$ we have
\begin{align*}
[\eta (-\Delta)^{\frac{\veps p}{2}}\vphi]_{\mathcal{X}^{s-\veps(p-1)}_{p}(\bbR^d)}^p &= \intdm{\bbR^d}{\intdm{\bbR^d}{\frac{\left| \left( \eta(\bx) (-\Delta)^{\frac{\veps p}{2}}\vphi(\bx) - \eta(\by) (-\Delta)^{\frac{\veps p}{2}}\vphi(\by) \right) \cdot{ \by-\bx \over |\by-\bx|} \right|^p}{|\bx-\by|^{d+(s-\veps (p-1))p}}}{\bx}}{\by} \\
&\leq C [(-\Delta)^{\frac{\veps p}{2}} \vphi]_{W^{s-\veps(p-1),p}(\bbR^d)}^p + C \intdm{\bbR^d}{\intdm{\bbR^d}{\frac{|\eta(\bx)-\eta(\by)|^p |(-\Delta)^{\frac{\veps p}{2}}\vphi(\bx)|^p}{|\bx-\by|^{d+(s-\veps(p-1))p}}}{\bx}}{\by} \\
&= I + II\,.
\end{align*}
We estimate the first term $I$ first. We will use the following identity that relates Riesz and  Bessel potentials, which can be found in \cite[Lemma 2,~Chapter V]{Stein}, that there exists a  finite measure $\mu$ that depends on $\veps$ and $p$ such that 
\[
(-\Delta)^{\frac{\veps p}{2}} \vphi = (1-\Delta)^{\frac{\veps p}{2}}(\vphi \, \ast \, \mu)\,.
\]
Using the fact that $(1-\Delta)^{\frac{\veps p}{2}} : W^{s+\veps,p}(\bbR^d) \, = \, W^{s-\veps(p-1)+\veps p,p}(\bbR^d) \to W^{s-\veps(p-1),p}(\bbR^d)$ is an isomorphism \cite[Theorem 4', Chapter V]{Stein}, we have
$$
I \leq C  \Vnorm{\vphi \, \ast \, \mu}_{W^{s+\veps,p}(\bbR^d)}^p\,.
$$
Now by Jensen's inequality and Fubini's theorem,
\begingroup\makeatletter\def\f@size{9.5}\check@mathfonts
\begin{align*}
\intdm{\bbR^d}{\intdm{\bbR^d}{\frac{\left| (\vphi \, \ast \, \mu)(\bx) - (\vphi \, \ast \, \mu)(\by) \right|^p}{|\bx-\by|^{d+(s+\veps)p}}}{\bx}}{\by} &= \intdm{\bbR^d}{ \intdm{\bbR^d}{\left| \intdm{\bbR^d}{\big( \vphi(\bz-\bx)-\vphi(\bz-\by)\big)}{\mu(\bz)} \right|^p \frac{1}{|\bx-\by|^{d+(s+\veps)p}}}{\bx}}{\by} \\
&\leq \mu(\bbR^d)^{p-1} \iiintdm{\bbR^d}{\bbR^d}{\bbR^d}{\frac{|\vphi(\bz-\bx)-\vphi(\bz-\by)|^p}{|\bx-\by|^{d+(s+\veps) p}}}{\mu(\bz)}{\bx}{\by} \\
&= \mu(\bbR^d)^p [\vphi]_{W^{s+\veps,p}(\bbR^d)}^p\,.
\end{align*}
\endgroup
Similarly, 
\[
\|\vphi \, \ast \, \mu\|^{p}_{L^{p}} \leq \mu(\bbR^{d})^{p}\|\vphi\|^{p}_{L^p} \leq C\mu(\bbR^{d})^{p} [\vphi]_{W^{s+\veps,p}(\bbR^d)}^p
\]
where the last inequality follows from a Poincar\'e-Korn type inequality, Corollary \autoref{Cor:P-korn}, where $C$ depends on the support set $4B$, $p$, and $\veps$. 
Combining the above two inequalities and using the fractional  Korn's inequality we proved we have that, 
\[I \leq C\,\mu(\bbR^{d})^{p} [\vphi]_{\mathcal{X}^{s+\veps}_{p}(\bbR^d)}^p\,.
\]
To estimate the second term $II$ we proceed as follows. 
\begin{align*}
II &\leq C \intdm{\complement(7B)}{\intdm{7 B}{\frac{|\eta(\bx)-\eta(\by)|^p |(-\Delta)^{\frac{\veps p}{2}} \vphi(\bx) |^p}{|\bx-\by|^{d+(s-\veps(p-1))p}}}{\bx}}{\by} + C \intdm{7B}{\intdm{7B}{\cdots}{\bx}}{\by} \\
&\leq C \Vnorm{\eta}_{L^{\infty}} \intdm{\complement(7B)}{\intdm{7 B}{\frac{|(-\Delta)^{\frac{\veps p}{2}} \vphi(\bx) |^p}{|\bx-\by|^{d+(s-\veps(p-1))p}}}{\bx}}{\by} + C \Vnorm{\grad \eta}_{L^{\infty}} \intdm{7B}{\intdm{7B}{\frac{|(-\Delta)^{\frac{\veps p}{2}} \vphi(\bx) |^p}{|\bx-\by|^{d+(s-\veps(p-1))p}}}{\bx}}{\by} \\
&\leq C \Vnorm{(-\Delta)^{\frac{\veps p}{2}} \vphi}^p_{L^p(\bbR^d)} \leq C \Vnorm{(-\Delta)^{\frac{\veps p}{2}} \vphi}_{W^{s-\veps(p-1),p}(\bbR^d)}^p\,.
\end{align*}
We repeat the argument used to bound $I$ and get that
\begin{align*}
II \leq C \Vnorm{\vphi}_{W^{s+\veps,p}(\bbR^d)}^p \leq C \mu(\bbR^d)^p [\vphi]_{\mathcal{X}^{s+\veps}_{p}(\bbR^d)}^p
\end{align*}
where again we have applied the Poincar\'e-Korn and fractional Korn inequalities. Thus \eqref{eq-TestFxnEstimate1} is proved.
\end{proof}

\begin{lemma}\label{lma-Density}
Let $\Omega \subset \bbR^d$ be an open set, and let $s \in (0,1)$, $p \geq 2$. Suppose $\bu \in \mathcal{X}^{s}_{p}(\bbR^d)$ such that $\supp \bu \Subset \Omega$. Then there exists a sequence $(\bu_n) \in C^{\infty}_c(\Omega;\bbR^d)$ with the property that 
\[
\bu_n \to \bu \quad \text{ in } \mathcal{X}^{s}_{p}(\bbR^d)
\] 
as $n \to \infty$.
\end{lemma}
\begin{proof} The sequence $\bu_n$ will be obtained via mollification.  Let $\phi \in C^{\infty}_c(\bbR^d)$ be a standard mollifier, i.e.
\[
\phi \geq 0\,, \quad \supp \phi \subset B_1(\textbf{0})\,, \quad \intdm{\bbR^d}{\phi(\bx)}{\bx} = 1\,.
\]
For any $\delta < {1\over 4}{\dist(\supp \bu, \complement\Omega)}$, introduce $\phi_{\delta} \in C^{\infty}_c(\bbR^d)$  by 
$
\phi_{\delta}(\bx) := \frac{1}{\delta^d} \phi \left( \frac{\bx}{\delta} \right)
$. Then take 
$$
\bu_{\delta} := \bu \, \ast \, \phi_{\delta} \in C^{\infty}_c(\Omega;\bbR^d)\,. 
$$
Extending the vector fields with value zero outside of $\Omega, $ we can easily show via basic estimates that $\|\bu_n - \bu\|_{\mathcal{X}^{s}_{p}(\bbR^{d})} \to 0 $ as $\delta\to 0$.
\end{proof}
The following is an adaptation to our setting of the interpolation lemma proved in \cite{Schikorra}.

\begin{lemma}\label{lma-InterpolationLemma}
Let $B \subset \bbR^d$ be a ball . Then for any $\delta>0$ and any  $\bu \in \mathcal{X}^{s}_{p}(4B)$, we have that 
\begin{align*}
[\bu]_{\mathcal{X}^{s}_{p}(B)}^p &\leq \delta^p [\bu]_{\mathcal{X}^{s}_{p}(4B)}^p 
 +  \frac{C}{\delta^{p'}} \left(\sup_{\phi} \langle \bbL^{s}_{p, 4B}{\bf u}, \phi\rangle\right)^{p\over p-1}
+C \frac{\text{diam}(B)^{-sp}}{\delta^{p(p-1)}} \intdm{4B}{|\bu(\bx)|^p}{\bx}\,,
\end{align*}
where the supremum is over all $\phi \in C^{\infty}_c(2B;\bbR^d)$ and $[\phi]_{\mathcal{X}^{s}_{p}(\bbR^d)} \leq 1$. The constant $C >0 $ depends only on $d$, $s$, $p$, and the ellipticity constants $\alpha_1$ and $\alpha_2$. Moreover if $s_0 >0$, and $s\in (s_0, 1)$, then the constant $C$ can be made dependent on $s_0$ instead of $s$.  
\end{lemma}

\begin{proof}
Let $\eta \in C^{\infty}_c(2B)$, $\eta \equiv 1$ in $B$ be the usual cutoff function in $2B$.

Define
$$
\psi(\bx) := \eta(\bx)\bu(\bx)\,, \qquad \vphi(\bx) := \eta^2(\bx)\bu(\bx)\,.
$$
Then, using Lemma \ref{tech-tools} item (3),
\begin{equation}\label{eq-InterpolationLemmaMollifierEstimate}
[\psi]_{\mathcal{X}^{s}_{p}(\bbR^d)} + [\vphi]_{\mathcal{X}^{s}_{p}(\bbR^d)} \leq C \Vnorm{\bu}_{\mathcal{X}^{s}_{p}(2B)}\,.
\end{equation}
By definition of $\psi$ and using the lower bound on $K$ we have
\begin{align*}
[\bu]_{\mathcal{X}^{s}_{p}(B)}^p &\leq \alpha_1 \intdm{B}{\intdm{B}{\left| \cD(\bu)(\bx,\by) \right|^p \frac{K(\bx,\by)}{|\bx-\by|^{d+sp}}}{\bx}}{\by} \\
	&\leq \alpha_2 \intdm{4B}{\intdm{4B}{\left|  \cD(\bu)(\bx,\by) \right|^{p-2} \left( \cD(\psi)(\bx,\by) \right)^2 \frac{K(\bx,\by)}{|\bx-\by|^{d+sp}}}{\bx}}{\by}\,.
\end{align*}
We now use the following algebraic identity: for $a$, $b$, $c$, $d$ real numbers 
\[
(ab-cd)^2 = (ab-cd)(a-c)b + ab(c-a)(b-d) + (a^2 b-c^2 d)(b-d)
\]
we can expand  $\cD(\psi)(\bx,\by)^{2}$ as
\[
\begin{split}
(\cD(\psi)(\bx,\by))^{2} &= (\eta(\bx)-\eta(\by))  (\cD(\psi)(\bx,\by)) \bu(\bx)\cdot {(\by-\bdx)\over |\by-\bx|} + (\eta(\by)-\eta(\bx)) \cD(\bu)(\bx,\by) \psi(\bx)\cdot {(\by-\bdx)\over |\by-\bx|}\\
& \qquad  + (\cD(\vphi)(\bx,\by))(\cD(\bu)(\bx,\by))\,.
\end{split}
\]
Now, by adding and subtracting the appropriate quantities and
splitting the integral accordingly, we have
$$
[\bu]_{\mathcal{X}^{s}_{p}(B)}^p \leq  \langle \bbL^{s}_{p, 4B}{\bf u}, \vphi\rangle + I_1 + I_2\,,
$$
where
\begin{align*}
I_1 &:= \Lambda \intdm{4B}{\intdm{4B}{\frac{\left| \cD(\bu)(\bx,\by) \right|^{p-2}  | \eta(\bx)-\eta(\by)| \, \left| \cD(\psi)(\bx,\by) \right| \left| \bu(\bx) \cdot \frac{\bx-\by}{|\bx-\by|} \right|}{|\bx-\by|^{d+sp}}}{\bx}}{\by}\,, \\
I_2 &:= \Lambda \intdm{4B}{\intdm{4B}{\frac{\left| \cD(\bu)(\bx,\by) \right|^{p-1} |\eta(\by)-\eta(\bx)| \, \left| \psi(\bx) \cdot \frac{\bx-\by}{|\bx-\by|} \right|}{|\bx-\by|^{d+sp}}}{\bx}}{\by}\,.
\end{align*}
Note that since the vector field $\vphi \in \mathcal{X}^{s,p}(\bbR^{d})$ and that  $\supp \vphi \Subset 2B$, by the density Lemma \ref{lma-Density} 
\[
 \langle \bbL^{s}_{p, 4B}{\bf u}, \vphi\rangle \leq  \sup_{\phi}\langle \bbL^{s}_{p, 4B}{\bf u}, \phi\rangle \|\vphi\|_{\mathcal{X}^{s}_{p}(\bbR^d)} \leq C \, \sup_{\phi}\langle \bbL^{s}_{p, 4B}{\bf u}, \phi\rangle \| \bu\|_{\mathcal{X}^{s}_{p}(4B)}
\]
where the supremum is over all $\phi \in C^{\infty}_c(2B;\bbR^d)$ and $[\phi]_{\mathcal{X}^{s}_{p}(\bbR^d)} \leq 1$.
As for $I_1$, using the estimate that  $\Vnorm{\grad \eta}_{L^{\infty}} \leq C \frac{1}{\text{diam}(B)}$, 
\begin{align*}
I_1
&\leq C\, \text{diam}(B)^{-1} \intdm{4B}{\intdm{4B}{\left| \cD(\bu)(\bx,\by) \right|^{p-2} \frac{\left| \cD(\psi)(\bx,\by) \right| |\bu(\bx)|}{|\bx-\by|^{d+sp-1}}}{\bx}}{\by}\,.
\end{align*}
Let $t_2 = 1-s$. Then $d+sp -1 = d+s(p-2)+s-t_2$, and H\"older's inequality with $q = \frac{p}{p-2}$, $q' = \frac{p}{2}$ implies that
\begin{align*}
I_1&\leq C\, \text{diam}(B)^{-1} \iintdm{4B}{4B}{\frac{|\cD(\bu)(\bx,\by)|^{p-2}}{|\bx-\by|^{d ( \frac{p-2}{p}) + s(p-2)}} \cdot \frac{|\cD(\psi)(\bx,\by)| |\bu(\bx)|}{|\bx-\by|^{d (\frac{2}{p}) + s - t_2}}}{\bx}{\by} \\
& \leq C\, \text{diam}(B)^{-1} [\bu]_{\mathcal{X}^{s}_{p}(4B)}^{p-2} \left( \iintdm{4B}{4B}{\frac{|\cD(\psi)(\bx,\by)|^{p/2} |\bu(\bx)|^{p/2}}{|\bx-\by|^{d+ \frac{sp}{2} - \frac{t_2 p}{2}}}}{\bx}{\by} \right)^{2/p} \,.
\end{align*}
Applying Cauchy-Schwarz to the last integral,
\[
I_1\leq C\, \text{diam}(B)^{-1} [\bu]_{\mathcal{X}^{s}_{p}(4B)}^{p-2} [\psi]_{\mathcal{X}^{s}_{p}(4B)} \left( \iintdm{4B}{4B}{\frac{|\bu(\bx)|^p}{|\bx-\by|^{d - t_2 p}}}{\bx}{\by} \right)^{1/p}\,.
\]
Thus, since $t_2 > 0$,
\[
\iintdm{4B}{4B}{\frac{|\bu(\bx)|^p}{|\bx-\by|^{d-t_2 p}}}{\bx}{\by} \leq C\, (\text{diam}(B))^{t_2 p} \Vnorm{\bu}_{L^p(4B)}^p.
\]
Using again \eqref{eq-InterpolationLemmaMollifierEstimate}, the final estimate of $I_1$ is
\[
I_1 \leq C\, \text{diam}(B)^{-s} \Vnorm{\bu}_{\mathcal{X}^{s}_{p}(4B)}^{p-1} \Vnorm{\bu}_{L^p(4B)}.
\]
The integral $I_2$ can be estimated the same way as $I_1$. Therefore,
$$
[\bu]_{\mathcal{X}^{s}_{p}(B)}^p \leq C\,\left( \sup_{\phi}\langle \bbL^{s}_{p, 4B}{\bf u}, \phi\rangle \| \bu\|_{\mathcal{X}^{s}_{p}(4B)} + \text{diam}(B)^{-s} \Vnorm{\bu}_{\mathcal{X}^{s}_{p}(4B)}^{p-1} \Vnorm{\bu}_{L^p(4B)}\right)\,,
$$
where the supremum is over all $\phi \in C^{\infty}_c(2B;\bbR^d)$ and $[\phi]_{\mathcal{X}^{s}_{p}(\bbR^d)} \leq 1$.
From the last estimate we apply Young's inequality to obtain the result and conclude the proof. 
\end{proof}

\subsection{Higher differentiability of solutions}
In this section we prove the second main result of the paper. Before presenting the proof, we state a commutator estimate that is an adaptation of the commutator estimate established in \cite{Schikorra}. The proof of the theorem is essentially identical to the result given in \cite{Schikorra}, and we omit it here. 
\begin{theorem}\label{thm-CommutatorEstimate}
Let $s \in (0,1)$, $\veps \in [0,1-s)$. Take $B \subset \bbR^d$ a ball or all of $\bbR^d$. Let $\bu \in \mathcal{X}^{s}_{p}(B)$ and $\vphi \in C^{\infty}_c(B;\bbR^d)$. For a certain normalizing constant $c$ depending on $s$, $p$, and $\veps$ denote the commutator
$$
R_{\veps}(\bu,\phi) := \langle\bbL^{s+\veps}_{p,B} \bu, \vphi\rangle  - c \, \langle \bbL^s_{p,B} \bu, (-\Delta)^{\frac{\veps p}{2}} \vphi\rangle\,.
$$
Then there exists a constant $C_{\veps} = C(s,p,\veps,n,\Lambda) >0$ such that
$$
|R_{\veps}(\bu,\vphi)| \leq C_{\veps} \,\veps \,[\bu]_{\mathcal{X}^{s+\veps}_{p}(B)}^{p-1} [\vphi]_{\mathcal{X}^{s+\veps}_{p}(\bbR^d)}\,.
$$
Moreover, $C_\veps$ is monotone increasing in $\veps$. That is $C_\veps\leq C_{\veps_0}$ for any $0<\veps<\veps_0$. 
\end{theorem}

\begin{proof}[Proof of Theorem \ref{Schikorra-theorem}]
Let $\Omega_{0}\Subset\Omega_1\Subset \Omega_2\Subset\Omega$ be given, and $\eta\in C_{c}^{\infty}(\Omega_1)$ such that $\eta=1$ in $\Omega_0$.  Let $\tilde{\bu} = \eta \bu$.  The proof of the theorem will be done in two steps. 

\textbf{Step 1.} In this step we establish that there exists $\veps_0 \in (0, 1-s)$ such that for any $0 < \veps < \veps_0$, we have 
\begin{equation}\label{first-step-estimate}
\|\tilde{\bu}\|^{p-1}_{\mathcal{X}^{s+\veps}_{p}(\bbR^d)} \leq C \left(\|\tilde{\bu}\|^{p-1}_{\mathcal{X}^{s}_{p}(\bbR^{d})} + \|\bbL^{s}_{p,\Omega_2} \tilde{\bu}\|_{[\mathring{\mathcal{X}}^{s-\veps(p-1)}_{p}(\Omega_{2})]^{\ast}} \right). 
\end{equation}
We apply the technique and the argument in \cite{Schikorra}.  First notice that since the support of $\tilde{\bu}$ is contained in $\Omega_1$, we have that $\|\tilde{\bu}\|^{p-1}_{\mathcal{X}^{s+\veps}_{p}(\bbR^d)} = \|\tilde{\bu}\|^{p-1}_{\mathcal{X}^{s+\veps}_{p}(\Omega)}$. 
Next, find finitely many balls $(B_k)_{k=1}^N \subset \Omega_2$ so that $\bigcup_{k=1}^N  B_k \supset \Omega_1$. We also may assume that $\bigcup_{k=1}^N  10 B_k \subset \Omega_2$. 
Then we have 
\begin{align*}
\|\tilde{\bu}\|^{p}_{\mathcal{X}^{s+\veps}_{p}(\Omega)}
&= \iintdm{\Omega}{\Omega_1}{\frac{|(\tilde{\bu}(\by)-\tilde{\bu}(\bx))\cdot {\by -\bx \over |\by-\bx|} |^{p}}{|\by-\bdx|^{d+(s+\veps)p}} }{\bx}{\by} + \iintdm{\Omega}{\Omega \setminus \Omega_1}{\cdots}{\bx}{\by} \\
&\leq \sum_{k=1}^N \iintdm{\Omega}{B_k}{\cdots}{\bx}{\by} + \iintdm{\Omega_0}{\Omega \setminus \Omega_1}{\frac{|\tilde{\bu}(\by)|^p}{|\bx-\by|^{d+(s+\veps)p}}}{\bx}{\by} \\
&\leq \sum_{k=1}^N \iintdm{2B_k}{B_k}{\cdots}{\bx}{\by} + \sum_{k=1}^N \iintdm{\Omega \setminus 2B_k}{B_k}{\cdots}{\bx}{\by} + C_\veps \Vnorm{\bu}_{L^p(\Omega)}^p \\
&\leq \sum_{k=1}^N [\tilde{\bu}]_{\mathcal{X}^{s+\veps}_{p}(2B_k)}^p + C(\veps)\Vnorm{\bu}_{L^p(\Omega)}^p\,.
\end{align*}
This is because the second term on the second line and the second term on the third line have disjoint support in the integrals. When using the constant $C(\veps)$ we are emphasizing that the constant depends on $\veps$, and it may also depend on other quantities.

Using Lemma \ref{lma-InterpolationLemma} and the fact that the union of the finite number of balls $B_k$ cover no more than $\Omega$, we get for any $\delta> 0$
$$
\|\tilde{\bu}\|^{p}_{\mathcal{X}^{s+\veps}_{p}(\Omega)} \leq  \delta^p \|\tilde{\bu}\|^{p}_{\mathcal{X}^{s+\veps}_{p}(\Omega)} + C({\delta}, \veps) \Vnorm{\bu}_{L^p(\Omega)}^p + C \sum_{k=1}^N \delta^{-p'} \left( \sup_{\phi} \langle \bbL^{s+\veps}_{p,8B_k}\tilde{\bu}, \phi\rangle \right)^{\frac{p}{p-1}}
$$
where the supremum is over all $\phi \in C^{\infty}_c(4B_k;\bbR^d)$ with $[\phi]_{\mathcal{X}^{s+\veps}_{p}(\bbR^d)} \leq 1$. Choosing $\delta$ sufficiently small we can estimate
$$
\|\tilde{\bu}\|^{p}_{\mathcal{X}^{s+\veps}_{p}(\Omega)} \leq C(\veps)\Vnorm{\bu}_{L^p(\Omega)}^p +  C \sum_{k=1}^N \left( \sup_{\phi} \langle \bbL^{s+\veps}_{p,8B_k}\tilde{\bu}, \phi\rangle \right)^{\frac{p}{p-1}}
\,,
$$
where the supremum is over all $\phi \in C^{\infty}_c(4B_k;\bbR^d)$ with $[\phi]_{\mathcal{X}^{s+\veps}_{p}(\bbR^d)} \leq 1$.
With Theorem \ref{thm-CommutatorEstimate}, adding and subtracting $\langle\bbL^s_{p,8B_k}\tilde{\bu}, (-\Delta)^{\frac{\veps p}{2}} \phi\rangle$ we can estimate by
\[
\begin{split}
\|\tilde{\bu}\|^{p}_{\mathcal{X}^{s+\veps}_{p}(\Omega)} &\leq C_1(\veps) \Vnorm{\bu}_{L^p(\Omega)}^p + \veps^{\frac{p}{p-1}}  \,C_2(\veps)\,\|\tilde{\bu}\|^{p}_{\mathcal{X}^{s+\veps}_{p}(\Omega)} \\
&+ C\sum_{k=1}^N \left( \sup \left\lbrace \left| \langle\bbL^s_{p,8B_k}\tilde{\bu}, (-\Delta)^{\frac{\veps p}{2}} \phi\rangle \right| \, : \, \phi \in C^{\infty}_c(4B_k;\bbR^d), [\phi]_{\mathcal{X}^{s+\veps}_{p}(\bbR^d)} \leq 1 \right\rbrace \right)^{\frac{p}{p-1}}\,.
\end{split}
\]
Notice from Theorem \ref{thm-CommutatorEstimate} that for  $\veps_0$ small enough, we have that $C_2(\veps) \leq C_{2}(\veps_0)$ for all $\veps \in (0, \veps_0)$ and therefore  we can absorb the right-hand side term involving $\|\tilde{\bu}\|^{p}_{\mathcal{X}^{s+\veps}_{p}(\Omega)}$ on the left-hand side for $\veps \in (0,\veps_0)$. The estimate becomes
\begingroup\makeatletter\def\f@size{9.75}\check@mathfonts
$$
\|\tilde{\bu}\|^{p}_{\mathcal{X}^{s+\veps}_{p}(\Omega)} \leq C(\veps) \Vnorm{\bu}_{L^p(\Omega)}^p + \sum_{k=1}^N \left( \sup \left\lbrace \left|\langle\bbL^s_{p,8B_k}\tilde{\bu}, (-\Delta)^{\frac{\veps p}{2}} \phi\rangle \right| \, : \, \phi \in C^{\infty}_c(4B_k;\bbR^d), [\phi]_{\mathcal{X}^{s+\veps}_{p}(\bbR^d)} \leq 1 \right\rbrace \right)^{\frac{p}{p-1}}\,.
$$
\endgroup
Now for a given $\phi \in C^{\infty}_c(4B_k;\bbR^d)$ such that $[\phi]_{\mathcal{X}^{s+\veps}_{p}(\bbR^d)} \leq 1$ and  for each $1 \leq k \leq N$ define
$$
\Phi = \rho_k (-\Delta)^{\frac{\veps p}{2}}\phi\,,
$$
where $\rho_k \in C^{\infty}_c(6B_k)$ and $\rho_k \equiv 1$ on $5B_k$. Then by Lemma \ref{tech-tools} item 2) we have that $\Phi \in C^{\infty}_c(6B_k)$, and 
\begin{equation}\label{eq-TestFxnEstimate1}
[\Phi]_{\mathcal{X}^{s-\veps(p-1)}_{p}(\bbR^d)} \leq C_k [\phi]_{\mathcal{X}^{s+\veps}_{p}(\bbR^d)} \leq C_k\,.
\end{equation}
Now, the disjoint support of $(1-\rho_k)$ and $\phi$ implies, via Lemma \ref{tech-tools} item 1), that 
\begin{equation}\label{eq-TestFxnEstimate2}
\begin{split}
\Vnorm{\grad \left( (1-\rho_k)(-\Delta)^{\frac{\veps p}{2}} \phi \right)}_{L^{\infty}(8B_k)} &\leq \Vnorm{-\grad \eta_k \cdot (-\Delta)^{\frac{\veps p}{2}} \phi}_{L^{\infty}(8B_k)} + \Vnorm{(1-\rho_k) (-\Delta)^{\frac{1+\veps p}{2}}\phi}_{L^{\infty}(8B_k)} \\
&\leq C_k \Vnorm{(-\Delta)^{\frac{\veps p}{2}} \phi}_{L^{\infty}(6B_k \setminus 5B_k)} + C_k \Vnorm{( -\Delta)^{\frac{1+\veps p}{2}}\phi}_{L^{\infty}(8B_k \setminus 5B_k)} \\
&{\leq} C \text{diam}(B_k)^{-d-\veps p} |4B_k|^{\frac{p-1}{p}} \Vnorm{\phi}_{L^p(4B_k)} \\
&\qquad + \text{diam}(B_k)^{-d-1-\veps p}|4B_k|^{\frac{p-1}{p}} \Vnorm{\phi}_{L^p(4B_k)} \\
&\leq C_k[\phi]_{\mathcal{X}^{s+\veps}_{p}(\bbR^d)}\,,
\end{split}
\end{equation}
where we have used  the Poincar\'e-Korn and fractional Korn inequalities.  
As a consequence of \eqref{eq-TestFxnEstimate1} and \eqref{eq-TestFxnEstimate2}, an application of H\"older's inequality gives us
\begin{align*}
\left| \langle\bbL^s_{p,8B_k}\tilde{\bu}, (-\Delta)^{\frac{\veps p}{2}}\phi - \Phi \rangle \right| &\leq \Vnorm{\grad \left( (1-\rho_k)(-\Delta)^{\frac{\veps p}{2}} \phi \right)}_{L^{\infty}(8B_k)} \iintdm{8B_k}{8B_k}{\frac{ \left| \cD(\bu)(\bx,\by) \right|^{p-1}}{|\bx-\by|^{d+sp-1}}}{\bx}{\by} \\
&\leq C [\tilde{\bu}]_{\mathcal{X}^{s}_{p}(8B_k)}^{p-1} \left( \iintdm{8B_k}{8B_k}{\frac{1}{|\bx-\by|^{d+p(s-1)}}}{\bx}{\by} \right)^{1/p} \leq C[{\tilde{\bu}}]_{\mathcal{X}^{s}_{p}(\bbR^{d})}^{p-1}\,.
\end{align*}
With the above, adding and subtracting $\langle\bbL^s_{p,8B_k}\tilde{\bu}, \Phi\rangle$ and using the properties of $\Phi$ shown above, our estimate becomes
$$
\|\tilde{\bu}\|^{p}_{\mathcal{X}^{s+\veps}_{p}(\Omega)} \leq C \Vnorm{\tilde{\bu}}_{\mathcal{X}^{s}_{p}(\bbR^d)}^p + \sum_{k=1}^N \left( \sup \left\lbrace \left| \langle\bbL^s_{p,8B_k}\tilde{\bu}, \psi\rangle \right| \, : \, \psi \in C^{\infty}_c(6B_k;\bbR^d), \, [\psi]_{\cX^{s-\veps(p-1)}_{p}(\bbR^d)} \leq C_k \right\rbrace \right)^{\frac{p}{p-1}}\,.
$$
Lastly, we need to transform the support of the operator $\cL$ from $8B_k$ to $\Omega_2$. Since $\supp \psi \subset 6B_k$, the disjoint support of the integrals gives (using H\"older's inequality and then the Poincar\'e and Korn inequalities on the $\psi$ integral)
\begin{align*}
|\langle\bbL^s_{p,8B_k}\tilde{\bu}, \psi\rangle - \langle\bbL^s_{p,\Omega_2}\tilde{\bu}, \psi\rangle| &\leq 2 \iintdm{\Omega \setminus 8B_k}{6B_k}{\frac{ \left| \cD(\bu)(\bx,\by) \right|^{p-1} \left| \cD(\psi)(\bx,\by) \right| }{|\bx-\by|^{d+sp}}}{\bx}{\by} \\
&\leq C_k[\tilde{\bu}]_{\mathcal{X}^{s}_{p}(\bbR^d)}^{p-1} [\psi]_{\mathcal{X}^{s-\veps(p-1)}_{p}(\bbR^d)}\,.
\end{align*}
Therefore,
\[
\begin{split}
\|\tilde{\bu}\|^{p}_{\mathcal{X}^{s+\veps}_{p}(\Omega)} &\leq C \Vnorm{\tilde{\bu}}_{\mathcal{X}^{s}_{p}(\bbR^d)}^p + \sum_{k=1}^N \left( \sup \left\lbrace \left| \langle\bbL^s_{p,\Omega_2}\tilde{\bu}, \psi\rangle\right| \, : \, \psi \in C^{\infty}_c(6B_k;\bbR^d), \, [\psi]_{\cX^{s-\veps(p-1)}_{p}(\bbR^d)} \leq C_k \right\rbrace \right)^{\frac{p}{p-1}} \\
&\leq C\left( \Vnorm{\tilde{\bu}}_{\mathcal{X}^{s}_{p}(\bbR^d)}^p  + \|\bbL^s_{p,\Omega_2}\tilde{\bu}\|^{p\over p-1}_{ [\mathring{\mathcal{X}}^{s-\veps(p-1)}_{p} (\Omega_2)]^{\ast}}\right)\,.
\end{split}
\]

\textbf{Step 2.} In this step we estimate the right-hand side of \eqref{first-step-estimate} in terms of the dual norm of $\bF$ and $\|\bu\|_{\mathcal{X}^{s}_{p}(\Omega)}$.  
By computation it is not difficult to show that $[\tilde{\bu}]_{\mathcal{X}^{s}_{p}(\bbR^d)} \leq C \|\bu\|_{\mathcal{X}^{s}_{p}(\Omega)}.$ Moreover, 
one can also prove that
\[
\|\bbL^s_{p,\Omega_2}\tilde{\bu}\|^{p\over p-1}_{ [\mathring{\mathcal{X}}^{s-\veps(p-1)}_{p} (\Omega_2)]^{\ast}} \leq \|\bF\|^{p\over p-1}_{ [\mathring{\mathcal{X}}^{s-\veps(p-1)}_{p} (\Omega)]^{\ast}} + \|\bu\|^{p}_{\mathcal{X}^{s}_{p}(\Omega)}.
\]
Indeed, this is possible to show by using the same argument as in \cite[Localization  Lemma]{Schikorra} and in fact prove that for any $t\in (2s-1, s)$ we have 
\[
\|\bbL^{s}_{p,\Omega_2} \tilde{\bu}\|_{[\mathring{\mathcal{X}}^{t}_{p}(\Omega_{2})]^{\ast}} \leq C \left(\|\bbL^{s}_{p,\Omega} \bu\|_{[\mathring{\mathcal{X}}^{t}_{p}(\Omega)]^{\ast}}   
+ \|\bu\|^{p-1}_{\mathcal{X}^{s}_{p}(\Omega)}\right). 
\]
The proof is complete.
\end{proof}

\appendix
\section{The Fourier Transform of the Poisson-Type Kernel $\mathbb{P}_t$}
Here we obtain the Fourier transform of the Poisson-type kernel $\mathbb{P}_t$ that has been used to establish relations between various Poisson integrals. 
Recall that the $(d+1) \times (d+1)$ matrix  $\mathbb{P}_t(\bx) = (\mathfrak{p}_{t}^{jk})$ has the form 
\[
	\begin{bmatrix}
	\widetilde{\bbP}_t(\bx) & \bP_t^{[d+1]}(\bx) \\
	\bP_t^{[d+1]}(\bx) & \mathfrak{p}^{d+1, d+1}_t(\bx) \\ 
	\end{bmatrix}\,.
\]
Here, $\widetilde{\bbP}_t(\bx) : \bbR^d \to \bbM_d(\bbR)$ is a  $d\times d$ matrix function given by 
\begin{equation}\label{ddpoissonmatrix}
\widetilde{\bbP}_t(\bx) = \frac{2(d+1)}{\omega_d} \frac{t}{(|\bx|^2 + t^2)^{\frac{d+3}{2}}} \bx \otimes \bx\,.
\end{equation}
The function $\bP_t^{d+1}: \bbR^d \to \bbR^d$ is a vector valued function, which we consider both a row and column vector, given by  
\begin{equation}\label{dpoissonvector}
\bP_t^{[d+1]}(\bx) = \frac{2(d+1)}{\omega_d} \frac{t^2 \bx}{(|\bx|^2 + t^2)^{\frac{d+3}{2}}}\,.
\end{equation} 
Finally the $(d+1)\times (d+1)$ entry is given by the function $\mathfrak{p}^{d+1,d+1}_t : \bbR^d \to \bbR$  defined as
\begin{equation}\label{poissonscalar}
\mathfrak{p}^{d+1,d+1}_t(\bx) = \frac{2(d+1)}{\omega_d} \frac{t^3 }{(|\bx|^2 + t^2)^{\frac{d+3}{2}}}\,.
\end{equation}
We compute the Fourier transform of each of these functions and put those transforms together to obtain the Fourier transform of  $\mathbb{P}_t$. 
We begin by writing some useful  Fourier transform formulas.  Rather than write the calculations explicitly in-line each time during a proof, we instead reference the formulas in their full generality. For completeness they are listed here and their proofs can be found in many textbooks, for example \cite{Mitrea-book}.  
\begin{itemize}
\item Let $n \in \bbN$, $\lambda \in (0,n)$, $f_{\lambda}(\bx) = |\bx|^{-\lambda}$, $\bx\in \mathbb{R}^{n} $. Then,
$
\cF \left( f_{\lambda} \right)(\bfxi) = \frac{\Gamma(\frac{n-\lambda}{2})}{\Gamma(\frac{\lambda}{2})} \pi^{\lambda-n/2} |\bfxi|^{\lambda-n}\,.
$

\item Let $n \in \bbN$, $n \geq 3$. Then for each $j$, $k \in \{ 1, 2, \ldots, n \}$ we have
\begin{equation}\label{Fourierformula1} 
\cF^{-1} \left( \frac{\xi_j \xi_k}{|\bfxi|^4} \right) = 4 \pi^2 \left[ \frac{1}{2(n-2)\omega_{n-1}} \cdot \frac{\delta_{jk}}{|\bx|^{n-2}} - \frac{1}{2 \omega_{n-1}} \cdot \frac{x_j x_k}{|\bx|^n} \right] \qquad \text{ in } \cS'(\bbR^n)\,,
\end{equation}
and as a consequence we have that 
\begin{equation}\label{Fourierformula2}
\cF \left( \frac{x_j x_k}{|\bx|^n} \right) = \frac{1}{4 \pi^2} \left[ \omega_{n-1} \frac{\delta_{jk}}{|\bfxi|^2} - 2 \omega_{n-1} \frac{\xi_j \xi_k}{|\bfxi|^4} \right] \qquad \text{ in } \cS'(\bbR^n)\,.
\end{equation}
\item 
For $a \in (0,\infty)$ and $x \in \bbR$, define $f(x) = \frac{1}{x^2 + a^2}$ and $g(x) = \frac{x}{x^2+a^2}$. Then
\begin{equation}\label{Fourierformula3}
\cF (f)(\xi) = \frac{\pi}{a} \e^{-2 \pi a |\xi|} \quad \text{ for every } \xi \in \bbR \text{ and in } \cS'(\bbR)\,,
\end{equation}
and
\begin{equation} \label{Fourierformula4}
\cF (g)(\xi) = - \pi \imath \, \sgn(\xi) \, \e^{-2 \pi a |\xi|} \quad  \text{ in } \cS'(\bbR)\,.
\end{equation}
\end{itemize}
We assume throughout that $d \geq 2$.
\begin{proposition}
For every $t>0$, we have the following:
\begin{itemize}
\item[1)] $
\cF_{\bx}(\widetilde{\bbP}_t)(\bfxi) = \e^{-2 \pi |\bfxi| t} \, \bbI_d - (2 \pi |\bfxi| t) \e^{-2 \pi |\bfxi| t} \left( \frac{\bfxi \otimes \bfxi}{|\bfxi|^2} \right) \quad \text{ in } \quad \cS'(\bbR^d)\,.$
\item[2)] $\cF_{\bx}(\bP_t^{[d+1]})(\bfxi)  = (2 \pi |\bfxi| t) \e^{-2 \pi |\bfxi| t} \left( -\imath \frac{\bfxi}{|\bfxi|} \right) \quad \text{ in } \quad \cS'(\bbR^d)\,.$
\item[3)] $\cF_{\bx}(\mathfrak{p}^{d+1,d+1}_t)(\bfxi) = \left( 1 + 2 \pi |\bfxi| t \right) \e^{-2 \pi |\bfxi| t} \quad \text{ in } \quad \cS'(\bbR^d)\,.$
\end{itemize}
\end{proposition} 
\begin{proof}[Proof of Item 1)]
Let $j,k \in \{ 1,2,\ldots, d\}$. Since $\widetilde{\bbP}_t = (\mathfrak{p}_t^{jk}) \in L^p(\bbR^d;\bbM_d(\mathbb{R}))$ for every $1 \leq p \leq \infty$  we have that $\mathfrak{p}_t^{jk} \in \cS'(\bbR^d)$ for every $t>0$. Thus, its Fourier transform  is a well-defined object in $\cS'(\bbR^d)$ and agrees with its Fourier transform as a function in $L^1(\bbR^d)$. The plan is to make use of partial Fourier transforms. Specifically, we will compute $\cF_{\bx}\left( \mathfrak{p}_t^{jk} \right)(\bfxi)$ in $\cS'(\bbR^d)$ by using the Fourier transform of ${\mathfrak{p}}^{jk}$ in $\cS'(\bbR^{d+1})$. To that end, we first compute the Fourier transform of ${\mathfrak{p}}^{jk}_t(\bx) = \frac{2(d+1)}{\omega_d} \frac{x_j x_k t}{(|\bx|^2 + t^2)^{d+3\over 2}}$ in $\cS'(\bbR^{d+1})$.  We use several properties of the Fourier transform. We denote the Fourier variables in $\mathbb{R}^{d+1}$ by $(\bfxi, \eta)$. 

Using the observation that  ${\mathfrak{p}}^{jk}_t(\bx) = -{2\over\omega_d} {\partial\over \partial t}\left( \frac{x_j x_k}{(|\bx|^2 + t^2)^{d+1\over 2}}\right)$, we have that  
\begin{equation}\label{eq-FourierTransformProofCase1-Equation1}
\begin{split}
\cF_{\bx,t}(\overline{\mathfrak{p}_t}^{jk})(\bfxi, \eta) = \frac{4 \pi \imath \eta}{\omega_d}\cF_{\bx,t} \left( \frac{x_j x_k}{(|\bx|^2 + t^2)^{d+1\over 2}} \right) &= - \frac{4 \pi \imath \eta}{\omega_d} \cdot \frac{1}{4 \pi^2} \left[ \omega_d \frac{\delta_{jk}}{|\bfxi|^2 + \eta^2} - 2 \omega_d \frac{(\bfxi, \eta)_j (\bfxi, \eta)_k}{(|\bfxi|^2 + \eta^2)^2} \right] \\
&= -\frac{\imath}{\pi} \eta \left[ \frac{\delta_{jk}}{|\bfxi|^2 + \eta^2} - 2 \frac{(\bfxi, \eta)_j (\bfxi, \eta)_k}{(|\bfxi|^2 + \eta^2)^2} \right]\,.
\end{split}
\end{equation}
where in the second equality we have applied the Fourier transform formula \eqref{Fourierformula2} with $n=d+1$.  
By taking partial inverse Fourier transform in $\eta$ and using the definition of $\mathbb{P}_t$, we see that for $j,k\in \{1, 2, \dots, d\}$
$$
\cF_\bx({\mathfrak{p}_t}^{jk})(\bfxi) = \cF^{-1}_{\eta} \big( \cF_{\bx,t}({\mathfrak{p}_t}^{jk})(\bfxi,\eta) \big)(t)
$$
for every $\bfxi \in \bbR^d \setminus \{ 0 \}$ and for every $t>0$. Our task is to compute the right hand side. Applying $\cF^{-1}_{\eta}$ to the right hand side of \eqref{eq-FourierTransformProofCase1-Equation1} we obtain for $j,k\in \{1, 2, \dots, d\}$ that 
\begin{equation}\label{eq-FourierTransformOfP-proof-equalityI}
\begin{split}
\cF^{-1}_{\eta} \left( - \frac{\imath}{\pi} \left[ \frac{\eta \delta_{jk}}{(|\bfxi|^2 + \eta^2)} - 2 \frac{\eta \xi_j \xi_k}{(|\bfxi|^2+\eta^2)^2} \right] \right) (t) &= \frac{-\imath}{\pi} \cF^{-1}_{\eta} \left( \frac{\eta \delta_{jk}}{(|\bfxi|^2 + \eta^2)} \right) (t) +\frac{2\imath}{\pi} \cF^{-1}_{\eta} \left( \frac{\eta \xi_j \xi_k}{(|\bfxi|^2+\eta^2)^2} \right) (t) \\
&= \frac{-\imath \delta_{jk}}{\pi} \cF_{\eta} \left( \frac{-\eta}{|\bfxi|^2+\eta^2} \right)(t) + \frac{2\imath \xi_j \xi_k}{\pi} \cF_{\eta} \left( \frac{-\eta}{(|\bfxi|^2+\eta^2)^2} \right)(t)\\
\end{split}
\end{equation}
where in the in the second inequality we used that $\cF(\cF(f))(\bx) = f(-\bx)$.  We will use the formula \eqref{Fourierformula4} to compute and simplify the first term of \eqref{eq-FourierTransformOfP-proof-equalityI} as 
\[
\frac{-i \delta_{jk}}{\pi} \cF_{\eta} \left( \frac{-\eta}{|\bfxi|^2+\eta^2} \right)(t) = \frac{i \delta_{jk}}{\pi} \left( -\pi i \sgn(t) \e^{-2 \pi |\bfxi| |t|} \right) = \delta_{jk} \sgn(t) \e^{-2 \pi |\bfxi| |t|}
\]
The second term in the expression \eqref{eq-FourierTransformOfP-proof-equalityI} can be computed using \eqref{Fourierformula3} and can be simplified as
\[
\begin{split}
\frac{2i \xi_j \xi_k}{\pi} \cF_{\eta} \left( \frac{-\eta}{(|\bfxi|^2+\eta^2)^2} \right)(t) =  \frac{2i \xi_j \xi_k}{\pi}  \cF_{\eta} \left( \frac{\mathrm{d}}{\mathrm{d}\eta}\left( \frac{1}{2(|\bfxi|^2+\eta^2)} \right) \right)(t) &= \frac{2i \xi_j \xi_k}{\pi} \cdot \frac{(2 \pi i t)}{2} \cF_{\eta} \left( \frac{1}{|\bfxi|^2 + \eta^2} \right)(t) \\
&=- 2 \xi_j \xi_k t \left( \frac{\pi}{|\bfxi|} \e^{-2 \pi |\bfxi| |t|} \right)\\
&=- \frac{\xi_j \xi_k}{|\bfxi|^2} (2 \pi |\bfxi| t) \e^{-2 \pi |\bfxi| |t|}\,.
\end{split}
\]
In the above calculation we have use the fact that $|\bfxi|$, $\frac{1}{|\bfxi|}$ are in $\cS'(\bbR^d)$ and that for any multi-index $\alpha$ and any number $k < |\alpha|$ the quantity $\frac{(\bfxi)^{\alpha}}{|\bfxi|^k} \e^{-|\bfxi|} \in L^1_{loc}(\bbR^d)$, hence in $\cS'(\bbR^d)$. 
Finally, plugging the last expressions into \eqref{eq-FourierTransformOfP-proof-equalityI} we obtain that for each $j,k \in \{ 1, 2, \ldots, d \}$
\begin{equation}
\cF(\mathfrak{p}_t)^{jk}(\bfxi) = \e^{-2 \pi |\bfxi| t} \delta_{jk} - (2 \pi |\bfxi'| t) \e^{-2 \pi |\bfxi| t} \, \frac{\xi_j \xi_k}{|\bfxi|^2}
\end{equation}
for every $\bfxi\in \bbR^d \setminus \{ 0 \}$ and $t > 0$.  
\end{proof}

\begin{proof}[Proof of \bf item 2)]This proof is much the same as the last one.
Let $j \in \{ 1,2,\ldots, d\}$. As before, we compute the Fourier transform of $\mathfrak{p}^{j,d+1}_t(\bx) = \frac{2(d+1)}{\omega_d} \frac{x_j t^2}{(|\bx|^2 + t^{2})^{d+3\over 2}}$ in $\cS'(\bbR^d)$. First notice again that 
\[
\mathfrak{p}^{j,d+1}_t(\bx) = {\partial \over \partial_{x_j}} \left[ {-2t^{2} \over \omega_d(|\bx|^2 + t^{2})^{d+1 \over 2}} \right]
\]
 Now we can use  \eqref{Fourierformula2} with $n = d+1$ to obtain 
\begin{equation}\label{eq-FourierTransformProofCase2-Equation1}
\begin{split}
\cF_{\bx, t}(\mathfrak{p}_t^{j,d+1})(\bfxi) = \frac{-2}{\omega_d} \cdot (2 \pi \imath \xi_j)\cF_{\bx, t} \left( \frac{t^{2}}{(|\bx|^2 + t^{2})^{d+1 \over 2}} \right) &= - \frac{4 \pi \imath \xi_j}{\omega_d} \cdot \frac{1}{4 \pi^2} \left[ \omega_d \frac{1}{|\bfxi|^2 + \eta^2} - 2 \omega_d \frac{\eta^2}{(|\bfxi|^2 + \eta^2)^{2}} \right] \\
&= -\frac{\imath}{\pi} \xi_j \left[ \frac{1}{|\bfxi|^2 + \eta^{2}} - 2 \frac{\eta^2}{(|\bfxi|^2 + \eta^2)^{2}} \right]\,.
\end{split}
\end{equation}
Again taking partial Fourier transforms we see that
$$
\cF( \mathfrak{p}_t ^{j, d+1})(\bfxi) = \cF^{-1}_{\eta} \big( \cF_{\bx,t}(\mathfrak{p}^{j,d+1})(\bfxi,\eta) \big)(t)
$$
for every $\bfxi\in \bbR^d \setminus \{ 0 \}$ and for every $t>0$. Our task is to compute the right hand side. Applying $\cF^{-1}_{\eta}$ to the right hand side of \eqref{eq-FourierTransformProofCase2-Equation1} we obtain
\begin{equation}\label{item2-inverse1}
\begin{split}
\cF^{-1}_{\eta} \left( - \frac{\imath}{\pi} \left[ \frac{\xi_j}{(|\bfxi|^2 + \eta^2)} - 2 \frac{\xi_j \eta^2}{(|\bfxi|^2+\eta^2)^2} \right] \right)  &= \frac{-\imath}{\pi} \cF^{-1}_{\eta} \left( \frac{\xi_j}{(|\bfxi|^2 + \eta^2)} \right)  + \frac{2\imath}{\pi} \cF^{-1}_{\eta} \left( \frac{\eta^2 \xi_j}{(|\bfxi|^2+\eta^2)^2} \right) \\
&= \frac{-\imath \xi_j}{\pi} \cF_{\eta} \left( \frac{1}{|\bfxi|^2+\eta^2} \right) + \frac{2i \xi_j}{\pi} \cF_{\eta} \left( \frac{\eta^2}{(|\bfxi|^2+\eta^2)^2} \right) \\
&= \frac{-\imath \xi_j}{\pi} \left( \frac{\pi}{|\bfxi|} \e^{-2 \pi |\bfxi| |t|} \right) + \frac{2\imath \xi_j}{\pi}  \frac{1}{-2 \pi \imath} \frac{\mathrm{d}}{\mathrm{d} t} \cF_{\eta} \left( \frac{\eta}{(|\bfxi|^2+\eta^2)^2} \right)\,.
\end{split}
\end{equation}
In the third equality we used the identities \eqref{Fourierformula3} and \eqref{Fourierformula4}. We also used the identity
$
\cF(xf)(\xi) = -\frac{1}{2 \pi i} \frac{\mathrm{d}}{\mathrm{d}\xi} \cF(f)(\xi)\,.
$
Now, just as in the proof of item 1) we can show that
\begin{equation} \label{item2-inverse2}
\cF_{\eta} \left( \frac{\eta}{(|\bfxi'|^2+\eta^2)^2} \right) (t) = - \pi \imath t \left( \frac{\pi}{|\bfxi'|} \e^{-2 \pi |\bfxi'| |t|} \right)\,.
\end{equation}
Now putting together \eqref{item2-inverse1} and \eqref{item2-inverse1} we obtain that 
\begin{align*}
\cF( \mathfrak{p}_t ^{j, d+1})(\bfxi)  &= -\imath \frac{\xi_j}{|\bfxi'|} \e^{-2 \pi |\bfxi| |t|} - \frac{\xi_j}{\pi^2} \frac{\mathrm{d}}{\mathrm{d}t} \left( \frac{-\pi^2 \imath t}{|\bfxi'|} \e^{-2 \pi |\bfxi| |t|} \right) \\
&= -\imath \frac{\xi_j}{|\bfxi|} \e^{-2 \pi |\bfxi| |t|} +\imath \frac{\xi_j}{|\bfxi|} \left( \e^{-2 \pi |\bfxi| |t|} - (2 \pi |\bfxi| t) \e^{-2 \pi |\bfxi| |t|} \right) \\
&= (2 \pi |\bfxi|t) \e^{-2 \pi |\bfxi| t} \, \left( -\imath \frac{\xi_j}{|\bfxi|} \right)\,.
\end{align*}

\end{proof}

\begin{proof}[Proof of {\bf item 3)}]
Writing 
$
\frac{t^3}{(|\bx|^2 + t^{2})^{d+3\over 2}} 
= \frac{t}{(|\bx|^2 + t^{2}){d+1\over 2}} - \sum_{j = 1}^d \frac{t x_j^2}{(|\bx|^2 + t^{2})^{d+3\over 2}}\,.
$
we notice that 
\begin{equation}
\mathfrak{p}^{d+1,d+1}_t(\bx) = \frac{2(d+1)}{\omega_d} \frac{t}{(|\bx|^2 + t^{2})^{d+1\over 2}} - \sum_{j=1}^d \mathfrak{p}_t^{jj}(\bx).
\end{equation}
That is, 
$\mathfrak{p}^{d+1,d+1}_t(\bx)= (d+1)p_t(\bx) -\sum_{j=1}^d \mathfrak{p}_t^{jj}(\bx)\,.$ 
Taking the Fourier Transform on both sides we get
\begin{equation}
\begin{split}
\cF_{\bx}(\mathfrak{p}^{d+1,d+1}_t) &= (d+1)\e^{-2 \pi |\bfxi| t} - \sum_{j=1}^d \left( \e^{-2\pi |\bfxi| t} - (2 \pi |\bfxi| t )\e^{-2 \pi |\bfxi| t} \frac{\xi_j^2}{|\bfxi|^2} \right) \\
&= (d+1)\e^{-2 \pi |\bfxi| t} - d \, \e^{- 2 \pi |\bfxi| t} + (2 \pi |\bfxi| t) \e^{-2 \pi |\bfxi| t} \\
&= (1 + 2 \pi |\bfxi| t) \e^{- 2 \pi |\bfxi'| t}\,, 
\end{split}
\end{equation}
as desired.

\end{proof}

\end{document}